\documentclass[11pt]{article}

\usepackage{amsfonts,amsmath, amsthm,latexsym,color,epsfig,mathrsfs, enumerate, xcolor}
\newtheorem{thm}{Theorem}[section]

\newtheorem{lem}[thm]{Lemma}

\newtheorem{conjecture}[thm]{Conjecture}

\newcommand{\A}{\mathcal{A}}
\newcommand{\B}{\mathcal{B}}

\newcommand{\D}{\mathcal{D}}
\newcommand{\E}{\mathcal{E}}
\newcommand{\F}{\mathcal{F}}
\newcommand{\G}{\mathcal{G}}
\newcommand{\I}{\mathcal{I}}
\newcommand{\J}{\mathcal{J}}

\newcommand{\mc}[1]{\mathcal{#1}}
\newcommand{\mb}[1]{\mathbb{#1}}
\newcommand{\nib}[1]{\noindent {\bf #1}}

\newcommand{\bfl}[1]{\left\lfloor #1 \right\rfloor}
\newcommand{\bcl}[1]{\left\lceil #1 \right\rceil}

\newcommand{\sub}{\subset}

\newcommand{\Ra}{\Rightarrow}
\newcommand{\sm}{\setminus}

\newcommand{\eps}{\varepsilon}
\newcommand{\es}{\emptyset}
\newcommand{\pl}{\partial}

\newcommand{\aA}{\alpha}

\newcommand{\gG}{\gamma}
\newcommand{\dD}{\delta}

\newcommand{\tT}{\theta}

\title{Stability for vertex isoperimetry in the cube}
\author{Peter Keevash\thanks{Mathematical Institute, University of Oxford, Oxford, UK. 
E-mail: keevash@maths.ox.ac.uk.\newline \hspace*{1.5em}
Research supported in part by ERC Consolidator Grant 647678.}
\and Eoin Long\thanks{Mathematical Institute, University of Oxford, Oxford, UK. 
E-mail: long@maths.ox.ac.uk.}}
\date{\today}

\newcommand{\Lov}[1]{\operatorname{B}_{lov}(|#1|)}
\newcommand{\LovDim}[2]{\operatorname{B}_{lov}^{(#1)}(|#2|)}

\begin{document}

\maketitle

\abstract{
We prove a stability version of Harper's cube vertex isoperimetric inequality,
showing that subsets of the cube with vertex boundary close 
to the minimum possible are close to (generalised) Hamming balls.
Furthermore, we obtain a local stability result for ball-like sets 
that gives a sharp estimate for the vertex boundary
in terms of the distance from a ball, and so our stability result
is essentially tight (modulo a non-monotonicity phenomenon).
We also give similar results for the Kruskal--Katona Theorem
and applications to new stability versions of some 
other results in Extremal Combinatorics.
}

\section{Introduction}

Isoperimetric inequalities have a long history in mathematics,
starting from the classical Euclidean isoperimetric inequality in $\mb{R}^d$ 
that balls minimise surface area among all sets with given volume.
There is also a rich theory of isoperimetric inequalities in the 
discrete setting, which has broad connections to a number of topics, 
including the concentration of measure phenomena, 
random graph and satisfiability thresholds and high-dimensional geometry.
This theory starts with the isoperimetric inequalities for 
the $n$-\emph{cube} $Q_n$, which is the graph on vertex set $\{0,1\}^n$ 
in which vertices are adjacent if they differ in a single coordinate. 
There are two natural notions of boundary for a set $\A \sub \{0,1\}^n$:
the \emph{vertex boundary} 
$\pl_v(\A) = \{x' \in \{0,1\}^n\sm \A: 
xx' \in E(Q_n) \mbox{ for some } x\in \A\}$
and the \emph{edge boundary} $\pl_e(\A) = 
\{xy \in E(Q_n): x \in \A , y \notin \A \}$.

This paper will be concerned with the vertex boundary,
for which the isoperimetric inequality was obtained by Harper \cite{Harper}.
To state his result, we define the simplical order on $\{0,1\}^n = \mc{P}[n]$
by $A < B$ if $|A| > |B|$ or $|A| = |B|$ 
and $\max(A \triangle B) \in B$.
We write $\I_m = \I^{(n)}_m$ 
for its initial segment of size $m$. 
Harper's theorem states that if $\A \subset \{0,1\}^n$ 
with $|\A| = m$ then $|\pl_v(\A)| \ge |\pl_v(\I_m)|$.
Given this inequality, it is natural to ask for which structures
equality holds (extremal configurations) 
or approximate equality holds (stability).
We are not aware of any results on these questions 
in the previous literature (by constrast, there are several 
such results \cite{Ellis, EKL, Friedgut, KKL, KL, Keller-Lifshitz} 
for the edge-isoperimetric inequality in the cube).

Our first result gives a stability result for Harper's theorem
for sets that have the same size as a Hamming ball
$\B = \B^n_{n-k}(C) :=
\{ A \sub \{0,1\}^n: |A \triangle C| \le n-k \}$;
here we note that all such balls have the same vertex-boundary 
(they can be identified by automorphisms of $Q_n$)
and if $m = \tbinom{n}{\ge k} := \sum_{i=k}^n \tbinom{n}{i}$ 
then $\I_m = \tbinom{[n]}{\ge k} := \cup_{i=k}^n \tbinom{[n]}{i}
= \B^n_{n-k}([n])$.

\begin{thm} \label{thm: ball-sized sets}
Suppose $\dD \in (0,1)$ and $\A \sub \{0,1\}^n$ 
with $|\A| = m = \tbinom{n}{\geq k}$ and 
$|\pl_v(\A)| \leq (1 + \tfrac{c}{n}) \tbinom{n}{k-1}$,
with $c = 10^{-3} \dD$. Then 
$|\A\triangle \mc{H}| \leq \dD \tbinom{n-1}{k-1}$ 
for some Hamming ball $\mc{H}$.
Furthermore, if $|\A\triangle \mc{H}|=2D$
then $|\pl_v(\A)| \ge |\pl_v(\J)|$
where $\J = \I_{m-D} \cup (\I_{m+D} \sm \I_m)$.
\end{thm}

\nib{Remarks.}
\begin{enumerate}[(i)]
\item Theorem \ref{thm: ball-sized sets}
is tight up to the value of the constant $c$.
For example, if $n=2k-1$ is odd then a `projected Hamming ball' 
$\A = \{A \sub [n]: |A \cap [n-2]| \ge k-1\}$
has size $|\A| = 2^{n-1} = \tbinom{n}{\ge k}$,
boundary $|\pl_v(\A)| = 4 \tbinom{n-2}{k-2}
= \tfrac{4(k-1)(n-k+1)}{n(n-1)}  \tbinom{n}{k-1}
= \big( 1 + \tfrac{1}{n} \big) \tbinom{n}{k-1}$
but $|\A\triangle \mc{H}| \ge \tbinom{n-1}{k-1}$
for any Hamming ball $\mc{H}$. 
\item The `furthermore' statement 
of Theorem \ref{thm: ball-sized sets} is a strong `local stability' 
result that gives a sharp estimate for the vertex boundary
in terms of the distance from a ball; it implies that
if the first statement holds with any value of $c$
then it in fact holds with an essentially optimal value.
In particular, we obtain uniqueness of the extremal configurations:
if $|\pl_v(\A)| = \tbinom{n}{k-1}$ then $\A$ is a Hamming ball.
\item It is tempting to guess that the local stability result
determines the exact dependence of $c$ on $\dD$,
i.e.\ the minimum possible value of $|\pl_v(\A)|$ 
over all $\A$ with $|\A| = m$ and given
$|\A\triangle \mc{H}| < \tbinom{n-1}{k-1}$.
Somewhat surprisingly, this is not true, 
as the minimum value of $|\pl_v(\A)|$
is not monotone in $|\A\triangle \mc{H}|$.
For example, if $n=5$ and $k=3$ (so $m=16$)
then for $D=0,1,2,3,4$ we have
$|\pl_v(\J)|=10,12,13,12,13$.
\end{enumerate}

As Theorem \ref{thm: ball-sized sets} describes the
stability of Harper's theorem for special values of $m=|\A|$, 
one will naturally ask next about general $m$, say 
$m = \tbinom{n}{\ge k+1} + m'$ with $0 \le m' < \tbinom{n}{k}$.
Here we note that if $\A = \tbinom{[n]}{\ge k+1} \cup \mc{C}$
where $\mc{C} \sub \tbinom{[n]}{k}$ with $|\mc{C}| = m'$
then $|\pl_v(\A)| = \tbinom{n}{k} - m' + |\pl \mc{C}|$,
where $\pl \mc{C} = \{B \in \tbinom {[n]}{k-1}: 
B \sub A \mbox { for some } A \in \mc{C} \}$
is the \emph{(lower) shadow} of $\mc{C}$.
By Harper's theorem, $|\pl_v(\A)| \ge |\pl_v(\I_m)|
= \tbinom{n}{k} - m' + |\pl \I^{(k)}_{m'}|$,
where $\I^{(k)}_{m'}$ is the initial segment of length $m'$
in the \emph{colex} order on $\tbinom {[n]}{k}$
(where $A < B$ if $\max(A \triangle B) \in B$).
Equivalently, $|\pl \mc{C}| \ge |\pl \I^{(k)}_{m'}|$,
which is the Kruskal--Katona theorem (see \cite{Kr,Ka}).
Thus a stability result for the Kruskal--Katona theorem
is a prerequisite for one in the general case of Harper's theorem.

The extremal configurations in the Kruskal--Katona theorem
were classified by F\"uredi and Griggs \cite{Furedi-Griggs} 
and independently by M\"ors \cite{Mors}. In the stability context, 
it is more convenient\footnote{
The exact function implicit in the Kruskal--Katona theorem
is rather pathological: Frankl, Matsumoto, Ruzsa and Tokushige \cite{KK-Takagi} 
proved that an appropriate rescaling converges to the Takagi function, 
which is continuous but nowhere differentiable.} 
to work with the following slightly weaker version 
of the Kruskal--Katona theorem due to Lov\'asz \cite{Lovasz}:
regarding $\tbinom{x}{k} = x(x-1)\dots (x-k+1)/k!$ 
as a polynomial in $x \in \mb{R}$, if $\A \sub \tbinom{[n]}{k}$
and $|\A| = \tbinom{x}{k}$ with $x\geq k$ 
then $|\pl(\A)| \geq \tbinom{x}{k-1}$. 
Keevash \cite{Keevash} gave a stability\footnote{
Our use of the term `stability' in this paper refers to
results that are also known as `99\% stability' results,
in that they describe structures that are very close to optimal.
In many contexts it is also interesting to describe some properties
of structures that are only within a constant factor of optimal;
such a `1\% stability' result for the Kruskal--Katona theorem
was given by O'Donnell and Wimmer \cite{OD-W}.}
version of this result, showing that for any $k \in \mb{N}$ and $\dD>0$ 
there is $\eps>0$ so that for $\A \sub \tbinom{[n]}{k}$,
if $|\A| =  \tbinom{x}{k}$ with $x\geq k$
and $|\pl (\A)| < (1+\eps) \tbinom{x}{k-1}$
then $|\A \triangle \tbinom{S}{k}| < \dD \tbinom{x}{k}$
for some $S \in \tbinom{[n]}{\bcl{x}}$.
Our next theorem concerns sets that are somewhat closer to 
a clique (with distance on a scale of $\tbinom{x-1}{k-1}$ 
rather than $\tbinom{x}{k}$), for which
we give a stronger stability result
with parameters that are tight
up to the value of the constant $c$.
Furthermore, as in Theorem \ref{thm: ball-sized sets},
we obtain a strong `local stability' result that gives a sharp estimate 
for the shadow boundary in terms of the distance from a clique,
which implies an essentially optimal dependence of parameters
(again with the non-monotonicity caveat).
In particular, this gives another proof 
for uniqueness of the extremal examples in the
Lov\'asz form of Kruskal--Katona,
i.e.\ if $|\pl(\A)| = \tbinom{x}{k-1}$ then $x \in \mb{N}$ 
and $\A=\tbinom{S}{k}$ for some $S \in \tbinom{[n]}{x}$.
For $S \sub [n]$ we define 
\begin{equation} \label{Jk}
\J^{(k)}_{|S|, E_1,E_2}:= 
\I^{(k)}_{\tbinom {|S|-1}{k} + E_1} \cup 
\big ( \I^{(k)}_{\tbinom {|S|}{k} + E_2} \sm 
\I^{(k)}_{\tbinom {|S|}{k}} \big ).
\end{equation}

\begin{thm} \label{thm: KK-stab}
If $\dD_0 \in (0,1)$ and $\A \sub \tbinom{[n]}{k}$ 
with $|\A| = \tbinom{x}{k}$ and 
$|\pl(\A)| \leq (1 + \tfrac{c}{x}) \tbinom{x}{k-1}$,
with $c = 10^{-9} \dD_0$, then 
$|\A\triangle \tbinom {S}{k}| 
\le \dD_0 \tbinom{|S|-1}{k-1}$ for some $S \sub [n]$.
Furthermore, if $|\A \cap {\tbinom {S}{k}}| 
= \tbinom{|S|-1}{k} + E_1$ and 
$|\A \sm {\tbinom {S}{k}}| = E_2$ 
where $0 \leq E_1, E_2 \leq \tbinom {|S|-1}{k-1}$ 
then $|\pl(\A)| \geq 
|\pl (\J^{(k)}_{|S|, E_1,E_2})|$.
\end{thm}

Now we return to the structural characterisation of
Harper's theorem for general sizes of the family $\A$.
Given the stability results in Theorems
\ref{thm: ball-sized sets} and \ref{thm: KK-stab},
one might conjecture a similar stability statement 
for initial segments of the simplicial order.
However, this is not true, as there is another
extremal configuration! Suppose 
$m = \tbinom{n}{\ge k+1} + \tbinom{s}{k}$ 
with $k \le s \le n$. 
Let \[ \G_1 = \tbinom{[n]}{\ge k+1} \cup \tbinom{[s]}{k} 
\quad \text{ and } \quad
\G_2 = \tbinom {[n]}{\geq k+1} 
\cup \tbinom{[s-1]}{k} \cup \tbinom{[s-1]}{k-1}.\] 
Then $\G_1 = \I_m$ is the initial segment of size $m$ 
in the simplical order, which is extremal by Harper's theorem.
Also, $|\G_2| = |\G_1| = m$ and
$|\pl_v(\G_2)| = \tbinom{n}{k} - \tbinom{s-1}{k} + \tbinom{s-1}{k-2}
=  \tbinom{n}{k} - \tbinom{s}{k} + \tbinom{s}{k-1} = |\pl_v(\G_1)|$.
Furthermore, if $s<n$ then $\G_1$ and $\G_2$ are not isomorphic.
We refer to $\G_1$ and $\G_2$ as generalised Hamming balls.
Our general stability result for Harper's theorem roughly says
that any family that is close to extremal must be close 
to a generalised Hamming ball. As for our stability result for 
Kruskal--Katona, our benchmark will be the corresponding Lov\'asz form 
of the vertex isoperimetric inequality:  if $\A \sub \{0,1\}^n$ with 
$|\A| = \tbinom {n}{\geq k+1} + \tbinom{x}{k}$ then
\begin{align}
\label{equation: coarse VIP bound}
|\pl_v(\A)| \geq \Lov{\A} := 
\tbinom {n}{k} - \tbinom {x}{k} + \tbinom {x}{k-1}.
\end{align}
Again, our parameters are essentially optimal,
as we also obtain a local stability result, 
with respect to the constructions
\[ \J_{m,D,E}:= \I_{m-D} \cup (\I_{m+E} \sm \I_m). \]

\begin{thm}
\label{thm: VIP stability}
Let $\dD \in (0,1)$, $c = 10^{-10} \dD$ and $\A \sub \{0,1\}^n$ 
with $|\A| = \tbinom{n}{\geq k+1} + \tbinom{x}{k}$ for some $k \ge 2$.
If $|\pl_v(\A)| \leq \Lov{\A} + \tfrac{ck(x-k)}{x^3} \tbinom{x}{k-1}$
then $|\A \triangle \G| \leq \dD \tbinom{x-3}{k-2}$ 
for some generalised Hamming ball $\G$.
Furthermore, writing $m=|\G |$, $D=|\G \sm A|$, $E=|\A \sm \G|$,
we have $|\pl_v(\A)| \ge |\pl_v(\J_{m,D,E})|$.
\end{thm}

Note that the assumption $k \ge 2$ in Theorem \ref{thm: VIP stability}
is necessary, as if $|\A| > \tbinom{n}{\geq 2}$ 
then $\pl_v(\A) = \{0,1\}^n \sm \A$, 
regardless of the structure of $\A$, so there is no stability.

We also give several applications of the above theorems to 
stability versions of other results in Extremal Combinatorics.
We start with the classical Erd\H{o}s-Ko-Rado theorem \cite{EKR},
that if $k \le n/2$ and $\A \sub \tbinom {[n]}{k}$ 
is intersecting ($A \cap B \neq \es $ for all $A,B \in \A$)
then $|\A| \le \tbinom{n-1}{k-1}$,
and if $k<n/2$ then equality holds only for a star
${\cal S}_i = \{A \in \tbinom {[n]}{k}: i\in A \}$.
There are many stability versions of this inequality in the literature (see 
\cite{BNR, DT, Dev-Kahn, EKL-EKR, Frankl-EKR-stab, Friedgut-EKR, Keevash}).

Here we will prove a tight stability result for intersecting families
with size sufficiently close to that of a star, which determines exactly
how large such a family can be in terms of the number $E$ of sets outside the star.
Given $E \le \tbinom {n-2}{k-1}$, we show that there is an extremal family 
$\F_E = \F^{out}_{E} \cup  \F^{in}_{E}$,
where $\F^{out}_{E}$ consists of the final $E$ sets 
of $\tbinom {[n] \sm \{1\}}{k}$ in colex order, 
and $\F^{in}_{E} \sub {\cal S}_1$ consists of all
sets in the star that intersect all sets in $\F^{out}_{E}$.
Note that as $E \le \tbinom {n-2}{k-1}$ all sets in $\F^{out}_{E}$
contain $n$, so $\F_E$ is intersecting.

\begin{thm} \label{EKR-stability}
Let $\tT \in (0,1/4)$, $c = 10^{-12}\tT$ 
and $n, k \in {\mathbb N}$ with $2k<n$.
Suppose $\A \subset \tbinom {[n]}{k}$ is intersecting.
If $|\A| \geq \big ( 1 - \tfrac {c(n-2k)}{n} \big ) \tbinom {n-1}{k-1}$ 
then there is a star ${\cal S}$ with
$E := |\A \sm {\cal S}| \leq 2\tT |{\cal S}|$.
Furthermore, $|\A| \leq |\F_E|$. 
In particular, if $E = \tbinom{u}{n-k-1}$ where $u \leq n-2$ 
then $|\A| \leq \tbinom {n-1}{k-1} - \tbinom {u}{k-1} + E$.
\end{thm}

\noindent {\emph{Remark:}} The upper bound on $|{\cal A}\setminus {\cal S}|$ above follows from Theorem 1.2 of Das and Tran in \cite{DT}.\vspace{1mm}

Next we consider a theorem of Katona \cite{Ka2} on families $\A \sub \{0,1\}^n$
that are $t$-intersecting ($|A \cap B| \ge t$ for all $A,B \in \A$).
For simplicity we just consider the case that $n+t=2k$ is even,
in which case Katona's theorem gives $|\A| \le \binom{n}{\ge k}$. If $t \ge 2$ then
equality holds only for the Hamming ball $\binom{[n]}{\ge k}$.
Here we prove a tight stability result for $t$-intersecting families
with size sufficiently close to that of a Hamming ball, which determines exactly
how large such a family can be in terms of the number $E$ of sets outside the ball.
Given $k,n \in \mb{N}$, $t=2k-n\geq 2$, $E \le \tbinom{n-1}{k-1}$, 
we show that there is an extremal family $\G_E$ obtained
from $\tbinom{[n]}{\ge k}$ by adding the 
initial $E$ elements of $\tbinom{[n]}{k-1}$ in colex
and deleting the final $E'$ elements of $\tbinom{[n]}{k}$ in colex,
where $E'$ is minimum subject to
$|\pl^{t-1}(I^{(k)}_{\tbinom{n}{k}-E'})| \le \tbinom{n}{n-k+1}-E$.

\begin{thm} \label{Katona-stability}
Let $k,n \in \mb{N}$ so that $k+t$ even, $t=2k-n \geq 2$,  
and $\tT = \min \{ 10^{-6} tn^{-1} e^{t^2/n}, 1 \}$ and $\dD \in (0,1/4)$. 
If $\A \subset \{0,1\}^n$ is $t$-intersecting
and $|\A| \geq \tbinom{n}{\ge k} - \tT\dD \tbinom{n-1}{k-1}$ 
then $E := |\A \sm \tbinom{[n]}{\ge k}| 
\le 5 \tT\dD \tbinom{n-1}{k-1}$,
so $|\A \triangle \tbinom{[n]}{\ge k}| 
\le 11 \tT\dD \tbinom{n-1}{k-1}$.
Furthermore, $|\A| \le |\G_E|$.
In particular, if $E = \tbinom{u}{k-1}$ where $u \leq n-1$ 
then $|\A| \leq \tbinom{n}{\ge k} + E - \tbinom {u}{n-k}$.
\end{thm}

For our final application we consider the Erd\H{o}s Matching Conjecture
(see \cite{ematch}) that the maximum size of $\A \sub \tbinom{[n]}{k}$
with no matching of size $t+1$ is achieved by $\tbinom{[tk+k-1]}{k}$ or
${\cal S}_T = \{A \in \tbinom{[n]}{k}: |A \cap T| \ne \es\}$
for some $T \in \tbinom{[n]}{t}$.
Ellis, Keller and Lifshitz \cite{EKL-EKR} showed how stability
for this problem can be deduced from isoperimetric stability.
(We thank Noam Lifshitz for drawing this to our attention
and suggesting that we might be able to obtain 
the improved bounds given here.)
Frankl \cite{Frankl-match} showed that the ${\cal S}_T$ are 
(uniquely) extremal for $n>(2t+1)k-t$. We will use this
to obtain the following stability result.

\begin{thm} \label{match-stability}
Let $\dD \in (0,1/4)$, $c = 10^{-10} \dD$
and $r,t,k,n \in \mb{N}$ with 
$r \le k$ and $n > (2t+1)(k+r)-t$. 
If $\A \sub \tbinom{[n]}{k}$ has no matching of size $t+1$ 
and $|\A| > \tbinom{n}{k} - (1 + \tfrac{rc}{n}) \tbinom{n-t}{k}$
then there is $T \in \tbinom{[n]}{t}$ such that 
$|\A \triangle {\cal S}_T| < 3\dD \tbinom {n-t-1}{k-1}$.
\end{thm}

The main new proof technique in our paper is a method 
for extracting stability results from compression arguments.
As far as we are aware, all known proofs of Harper's Theorem use
some form of compression, i.e.\ replacing any family by a sequence
of successively `simpler' families of the same size without
increasing the vertex boundary. One can prove Harper's Theorem by
showing that there is such a sequence that transforms any family 
into an initial segment of the simplicial order. 
As it applies to any family, it may at first seem hopeless
to obtain any structural information from this process.
However, for a suitably gradual sequence of transformations,
we are able to use the property of having small vertex boundary
to keep track of the structure of families under the reversal
of the compressions. A key tool in this analysis is 
a local stability result showing that sets with small vertex boundary 
that are reasonably close to an extremal example
must in fact be very close to an extremal example;
thus we can rule out a possible cumulative effect 
of a sequence of small adjustments from the compressions.

The organisation of this paper is as follows.
In the next section we collect various technical estimates
concerning binomial coefficients that will be used throughout the paper.
We prove stability for Kruskal--Katona in section 3
and for Harper's Theorem in section 4.
The applications are given in section 5,
and the final section contains some concluding remarks.

\medskip

\noindent \textbf{Notation.} 
We write ${\cal P}(S)$ for the power set 
(set of subsets) of a set $S$. 
Throughout we identify ${\cal P}[n]$ with $\{0,1\}^n$, 
where a set $A$ corresponds to its characteristic vector. 
We also write $\binom {S}{k} = \{A \subset S: |A|=k\}$. 
The complement of $A \sub [n]$ is $A^c:= [n] \sm A$. 
For $x \in A$ we write $A-x = A \sm \{x\}$.
For $x \in A^c$ we write $A+x = A \cup \{x\}$.
Given integers $m< n$ we write $[m,n] := \{m,m+1,\ldots ,n\}$ and let $[n]:= [1,n]$. 
We let $a\pm b$ denote some unspecified 
real number between $a-b$ and $a+b$.

\section{Estimates}

This section contains various properties of and estimates for
binomial coefficients that will be used throughout the paper.
We start by stating some simple formulae and inequalities 
for easy reference, which will henceforth be used without comment:
\begin{gather*}
\tbinom{x}{k} = \tbinom{x-1}{k} + \tbinom{x-1}{k-1}, \qquad
\tbinom{x-1}{k} \tbinom{x}{k}^{-1} = \tfrac{x-k}{x}, \qquad
\tbinom{x}{k-1} \tbinom{x}{k}^{-1} = \tfrac{k}{x-k+1}, \\
\tbinom{x-1}{k-1} \tbinom{x}{k}^{-1} = \tfrac{k}{x}, \qquad
\tbinom{x-2}{k-1} \tbinom{x}{k}^{-1} 
= \tfrac{k(x-k)}{x(x-1)} \le \tfrac{x}{4(x-1)}, \\
\tbinom{x-1}{k-1}
= \tfrac{(x-1)(x-2)}{(k-1)(x-k)} \tbinom{x-3}{k-2}, \qquad
\tfrac{k(x-k)}{x^2} \tbinom{x-1}{k-1}
= \tfrac{k(x-1)(x-2)}{(k-1)x^2} \tbinom{x-3}{k-2}
\le 2\tbinom{x-3}{k-2} \ \text{ if } \ x \ge k+1.
\end{gather*}

Next we give two lemmas concerning 
approximations of $\tbinom{x}{k}$ by $\tbinom{y}{k}$.
We omit the straightforward proof of the first of these.

\begin{lem} \label{binomratio}
For $x \ge y > k-1$ 
we have $(\tfrac{x}{y})^k 
\le \tbinom{x}{k} \tbinom{y}{k}^{-1} 
= \prod_{i=0}^{k-1} \tfrac{x-i}{y-i}
\le (\tfrac{x-k+1}{y-k+1})^k$. Therefore
\begin{enumerate}[(i)]
\item if $y>k-1$ and $x \ge (1+\tT)y$ with $\tT \ge 0$
then $\tbinom{x}{k} \ge (1+\tT)^k \tbinom{y}{k}$,
\item if $y \ge (1+\aA)k$ with $\aA>0$
and $\tbinom{x}{k} \ge (1+\tT) \tbinom{y}{k}$ 
with $\tT \in [0,1]$ then $x \geq \big ( 
{1 + \tfrac {\aA \tT }{2k(1 + \aA )}} \big ) y$.
\end{enumerate}
\end{lem}

\begin{lem} \label{shiftapprox}
Suppose $k \ge 2$, $x \ge y > k-1$, $0<c<1/2$ 
and $\tbinom{x}{k} = (1 \pm c) \tbinom{y}{k}$.
Then $\tbinom{x-1}{k-1} = (1 \pm c) \tbinom{y-1}{k-1}$,
$\tbinom{x-1}{k} =  
(1 \pm \tfrac{y+k}{y-k} c ) \tbinom{y-1}{k}$ (if $y>k$) 
and $\tbinom{x+1}{k} = (1 \pm c) \tbinom{y+1}{k}$.
\end{lem}

\begin{proof}
Note that $\tbinom{x-1}{k-1} \tbinom{y-1}{k-1}^{-1}
= \prod_{i=1}^{k-1} \tfrac{x-i}{y-i} \ge 1$ as $x \ge y$,
and $\tbinom{x-1}{k-1} \tbinom{y-1}{k-1}^{-1}
= \tbinom{x}{k} \tbinom{y}{k}^{-1}
\tfrac{y}{x} \le 1+c$. We deduce
$\tbinom{x-1}{k} =  \tbinom{x}{k} - \tbinom{x-1}{k-1}
= (1 \pm c) \tbinom{y}{k} - (1 \pm c) \tbinom{y-1}{k-1}
= \tbinom{y-1}{k} \pm c \tbinom{y}{k} \pm c \tbinom{y-1}{k-1}
= (1 \pm \tfrac{y+k}{y-k} c ) \tbinom{y-1}{k}$. Similarly,
$\tbinom{x+1}{k} \ge \tbinom{y+1}{k}$ and
$\tbinom{x+1}{k} \tbinom{y+1}{k}^{-1}
= \tbinom{x}{k} \tbinom{y}{k}^{-1}
\tfrac{(x+1)(y+1-k)}{(y+1)(x+1-k)} \le 1+c$
as $(x+1)(y+1-k) \le (y+1)(x+1-k)$.
\end{proof}

The remainder of this section is mostly concerned
with properties of the following functions.
For $k\in {\mathbb N}$ we define 
$f_k:[0,\infty ) \to [1,\infty)$ and
$g_k:(k-1,\infty) \to (0,\infty)$ by
\[  f_k(\tbinom {x}{k}) = \tbinom{x}{k-1}
\ \text{ for } \ x \ge k-1
\quad \text{ and } \quad
g_k(x) = \sum_{i=0}^{k-1} (x-i)^{-1}.\]
Note that $f_1(t)=1$ for all $t \ge 0$
and $\tbinom {x}{k} g_k(x)$ is the derivative
of $\tbinom {x}{k}$ with respect to $x$.
As $\tbinom{x}{k} g_k(x) \ge 
\tbinom{x}{k} \tfrac{k}{x} = \tbinom{x-1}{k-1}$,
by the Mean Value Theorem we have
\begin{equation}\label{mvt}
\tbinom{x+c}{k} \ge \tbinom{x}{k} + c \tbinom{x-1}{k-1}
\qquad \text{ for all } c \ge 0.
\end{equation}

The most important feature of $f_k$ for our purposes
is that it is concave, and that we have an effective
estimate for its second derivative, as follows.

\begin{lem} \label{deriv}
If $k \ge 2$, $x > k-1$ and $t = \tbinom {x}{k}$ then
\[ f_k'(t) = \frac{k g_{k-1}(x)}{(x-k+1)g_k(x)}, \qquad
f_k''(t) = \frac{k ( g_{k-1}'(x) - g_{k-1}(x)^2 )}{
t (x-k+1)^2 g_{k}(x)^3}, \]
and if $x \ge k-1+\aA$ with $\aA>0$ then
$-f_k''(t) > ((2+\aA^{-1})^2 (x-k+1)t)^{-1}$.
\end{lem} 

\begin{proof}
Differentiating the identity $f_k(t) = \tbinom{x}{k-1}$ 
with respect to $x$ gives 
$f_k'(t) t g_k(x) = \tbinom {x}{k-1} g_{k-1}(x)$,
and hence the stated formula for $f_k'(t)$.
Substituting $g_{k-1}(x) = g_k(x) - (x-k+1)^{-1}$ 
and differentiating again gives 
\[ f_k''(t) t g_k(x) =  \frac {k \big ( 2 g_{k}(x) 
+ (x-k+1) ( g_{k}'(x) - g_{k}(x)^2 ) \big ) }{ (x-k+1)^3 g_k(x)^2}. \]
To deduce the stated formula for $f_k''(t)$ we need to show 
\[ 2 g_{k}(x) + (x-k+1) ( g_{k}'(x) - g_{k}(x)^2 ) 
= (x-k+1) ( g_{k-1}'(x) - g_{k-1}(x)^2 ). \]
Using $g_{k-1}'(x) = g_{k}'(x) + (x-k+1)^{-2}$
and $g_{k}(x)^2 - g_{k-1}(x)^2 = (g_{k}(x)-g_{k-1}(x))
(g_{k}(x)+g_{k-1}(x)) = (x-k+1)^{-1}(g_{k}(x)+g_{k-1}(x))$
reduces this to the identity 
$g_{k}(x) = g_{k-1}(x) + (x-k+1)^{-1}$,
so the formula is valid. To see the final statement,
we first note that
$g_{k-1}'(x) = - \sum _{i=0}^{k-2} (x-i)^{-2} < 0$	 
since $k\geq 2$
and $(x-k+1)^{-1} \leq (1 + \aA^{-1})(x-k+2)^{-1}$,
so $g_{k}(x) \leq (2 + \aA^{-1})g_{k-1}(x)$. Thus
\[- f_{k}''(t) > k ((2+\alpha ^{-1})^2(x-k+1)^2g_k(x) t)^{-1},\]
which with $(x-k+1)g_{k}(x) \leq k$ gives the required bound.
\end{proof}

Next we record a simple consequence 
of the concavity shown in the previous lemma.

\begin{lem} \label{concavecor}
Suppose $x \ge \ell \ge 2$ and
$0 \le z \le \tbinom{x-1}{\ell-1}$. Then 
$q(z) := f_\ell(\tbinom{x}{\ell}-z) + f_{\ell-1}(z)
\ge \tbinom{x}{\ell-1}$.
\end{lem}

\begin{proof}
Note that $q$ is concave by Lemma \ref{deriv}
with $q(\tbinom{x-1}{\ell-1}) = 
\tbinom{x-1}{\ell-1} + \tbinom{x-1}{\ell-2}
= \tbinom{x}{\ell-1}$
and $q(0) = \tbinom{x}{\ell-1} + 1$.
The lemma follows.
\end{proof}

In the following two lemmas we show how an estimate
for the second derivative of a concave function $f$
translates into an effective estimate for
certain differences of the form
$(f(y)+f(z)) - (f(a)+f(b))$ where
$a \le y \le z \le b$ with $y+z=a+b$.

\begin{lem} \label{taylor}
Suppose $g:[a,b] \to \mb{R}$ is concave and non-negative 
and $-g''(t) \ge m$ for $t \in [a,c]$ with $c=a+w \le (a+b)/2$.
Then $g(a+d) \ge dwm/4$ for $d \in [0,c-a]$. 
\end{lem}

\begin{proof}
By Taylor's theorem, we have $a \le t_1 \le c \le t_2 \le b$ with
\begin{align*}
0 & \leq g(a) = g(c) - w g'(c) + \tfrac{1}{2} w^2 g''(t_1),\ \text{ and } \\
0 & \leq g(b) = g(c) + (b-c) g'(c) + \tfrac{1}{2} (b-c)^2 g''(t_2),\ \text{ so } \\
0 & \leq (b-c) \big( g(c) - w g'(c) + \tfrac{1}{2} w^2 g''(t_1) \big) \\
& \ \ + w \big( g(c) + (b-c) g'(c) + \tfrac{1}{2} (b-c)^2 g''(t_2) \big) \\
& \le (b-a) g(c) +  \tfrac{b-a}{4} w^2 g''(t_1) \le (b-a)( g(c) - w^2m/4 ). 
\end{align*}  
By concavity, $g(a+d) \ge \tfrac{d}{w} g(c) \ge dwm/4$, as required.
\end{proof}

\begin{lem} \label{defect}
Let $f:[a,b] \to \mb{R}$ be concave
with $-f''(t) \ge m$ for $t \in [a,a+w]$ 
with $w \le (b-a)/2$.
Suppose $a \le y \le z \le b$ with $y+z=a+b$
and $f(y)+f(z)<f(a)+f(b)+\Phi$.
Then $y-a \le 4\Phi/mw$.
\end{lem}

\begin{proof}
Define $g(t)=f(t)-h(t)$ where $h$ is the linear function
with $h(a)=f(a)$ and $h(b)=f(b)$.
Then $g$ is concave and non-negative,
$g(a)=g(b)=0$, $g''(t)=f''(t)$ and $g(y) \le g(y)+g(z) < \Phi$.
Also $g(y) \ge (y-a)wm/4$ by Lemma \ref{taylor},
so $y-a \le 4\Phi/mw$.
\end{proof}

Now we state some specific instances of
Lemma \ref{defect} (using Lemma \ref{deriv})
that will be used later in the paper.

\begin{lem} \label{defectapps} Let $n \ge x \ge \ell \ge 2$.
\begin{enumerate}[(i)]
\item Suppose $y, z \in {\mathbb N}$ satisfy $0 \le y \le z \le \tbinom{n}{\ell}$, 
with $\tbinom{n}{\ell} \ge y+z = X 
\ge \tbinom{n}{\ell} - \tfrac{1}{4} \tbinom{x}{\ell}$ and 
$f_\ell(y) + f_\ell(z) < 1 + f_\ell(X) + \tfrac{c}{x} \tbinom{x}{\ell-1}$.
Then $y \le 400 c \tbinom{x-1}{\ell-1}$.
\item Suppose $y,z \in {\mathbb N}$ with $0 \le y \le z \le \tbinom{n}{\ell}$,
with $y+z = \tbinom{n}{\ell} + E$, where
$0 < E < \tfrac{1}{4} \tbinom{x}{\ell}$,
and $f_\ell(y) + f_\ell(z) < f_\ell(E) + \tbinom{n}{\ell-1}
+ \tfrac{c}{x} \tbinom{x}{\ell-1}$.
Then $y \le E + 400 c \tbinom{x-1}{\ell-1}$.
\item Suppose $(1+\tT)\tbinom{x}{\ell} \le \tbinom{n}{\ell}$,
$\tbinom{x}{\ell} \le y \le z \le \tbinom{n}{\ell}$
with $y+z = \tbinom{x}{\ell} + \tbinom{n}{\ell}$
and $f_\ell(y) + f_\ell(z) <  \tbinom{x}{\ell-1}
+ \tbinom{n}{\ell-1} + \tfrac{c}{x} \tbinom{x}{\ell}$.
Then $y \le \tbinom{x}{\ell} + 72 c\tT^{-1} \tbinom{x-1}{\ell-1}$.
Furthermore, if $\tbinom{x}{\ell} 
< \tbinom{n-1}{\ell} + \tfrac{1}{2} \tbinom{n-1}{\ell-1}$
and $f_\ell(y) + f_\ell(z) <  \tbinom{x}{\ell-1}
+ \tbinom{n}{\ell-1} + \tfrac{c' \ell(x-\ell)}{x^3} \tbinom{x}{\ell-1}$
then $y \le \tbinom{x}{\ell} + 250c' \tbinom{x-3}{\ell-2}$.
\end{enumerate}
\end{lem}

\begin{proof}
For (i), let $a = \frac {1}{2}$, $b = X-\frac {1}{2}$ and note that
$1+f_\ell(X) \le f_\ell(a)+f_\ell(b)$ by concavity,
so $f_\ell(y) + f_\ell(z) < f_\ell(a)+f_\ell(b) + \Phi$,
where $\Phi = \tfrac{c}{x} \tbinom{x}{\ell}$.
Applying Lemma \ref{defect} with 
$w = \tfrac{1}{3}\tbinom{x}{\ell}$ 
and $m = (16(x-\ell+1)\tbinom{x}{\ell})^{-1}$
(by Lemma \ref{deriv} with $\aA=\frac {1}{2}$)
gives $y-a \le 4\Phi/mw
\le 4 \tfrac{c}{x} \tbinom{x}{\ell-1} \cdot 48(x-\ell+1)
\leq 200 c \tbinom{x-1}{\ell-1}$. Now if $\frac {1}{2} \leq 200 c \tbinom {x-1}{\ell -1}$ this gives $y \leq 400 c \tbinom {x-1}{\ell -1 }$. Otherwise $y < \frac {1}{2} + 200 c \binom {x-1}{\ell -1} < 1$ giving $y = 0 < 400 c \binom {x-1}{\ell -1}$ by integrality.
The proof of (ii) is the same,
using $b = \tbinom{x}{\ell} + E - a$.
Similarly, for (iii), applying Lemma \ref{defect} with 
$a = \tbinom{x}{\ell}$, $b = \tbinom{n}{\ell}$,
$\Phi = \tfrac{c}{x} \tbinom{x}{\ell}$,
$w = \tfrac{\tT}{2}\tbinom{x}{\ell}$ and
$m = (9(x-\ell+1)\tbinom{x}{\ell})^{-1}$ (taking $\alpha =1$) gives
$y-a \le 72 c\tT^{-1} \tbinom{x-1}{\ell-1}$.
For the `furthermore' statement, we apply this bound
with $c = \tfrac {c'\ell(x-\ell)}{x^3} \tbinom {x}{\ell-1} 
x \tbinom {x}{\ell}^{-1} \le \tfrac {c'\ell^2}{x^2}$, noting that 
$\tbinom{n}{\ell} \ge \tbinom{x}{\ell} + \tfrac{1}{2} \tbinom{n-1}{\ell-1}
= \tbinom{x}{\ell} + \tfrac{\ell}{2(n-\ell)} \tbinom{n-1}{\ell}
\ge (1+\tT) \tbinom{x}{\ell}$ with $\tT = \tfrac{\ell}{3(x-\ell)}$.
\end{proof}

Next we give a similar statement to that of the previous lemma
for certain sums involving both $f_k$ and $f_{k-1}$.

\begin{lem} \label{defectapp2}
Let $x \geq k \ge 3$, $X = \tbinom {x-1}{k}$ 
and $Y = \tbinom {x-1}{k-1}$.
Suppose $0 \le y \le Y$ with
$f_k(X+y) + f_{k-1}(Y-y) < \tbinom{x}{k-1} 
+  \tfrac{c}{x} \tbinom{x}{k-1}$.
Then $y \notin [600 cY, (1-600c)Y]$.
Furthermore, if $x \geq k+1$ then 
$y \notin [10^7 c \tbinom{x-2}{k-1}, (1-600c)Y]$.
\end{lem}

\begin{proof}
Let $e_k:[X,X+Y] \to \mb{R}$ and $e_{k-1}:[0,Y] \to \mb{R}$ 
be the linear functions with 
$e_k(X)=\tbinom{x-1}{k-1}$, $e_k(X+Y)=\tbinom{x}{k-1}$,
$e_{k-1}(0)=0$ and $e_{k-1}(Y)=\tbinom{x-1}{k-2}$. Note that
\[ e_k(X+y) + e_{k-1}(Y-y) 
= \tbinom{x-1}{k-1} + \tfrac{y}{Y} 
(\tbinom{x}{k-1} - \tbinom{x-1}{k-1})
+ (1 - \tfrac{y}{Y}) \tbinom{x-1}{k-2}
= \tbinom{x}{k-1}. \]
Let $h_k=f_k-e_k$ and $h_{k-1}=f_{k-1}-e_{k-1}$.
Then $h_k$ and $h_{k-1}$ are concave and non-negative,
with $h_k(X)=h_k(X+Y)=h_{k-1}(Y)=0$, $h_{k-1}(0)=1$ and 
$h_k(X+y) + h_{k-1}(Y-y) < \tfrac{c}{x} \tbinom{x}{k-1}$.

Next note that $\tbinom{k-2+1/4}{k-1} \le 1/4 \le Y/4$,
so Lemma \ref{deriv} with $\aA=1/4$ gives
$-h_{k-1}''(t) \ge m = (18(x-k+1)Y)^{-1}$ for $t \in [Y/4,Y/2]$.
Applying Lemma \ref{taylor} with $a = d = w = Y/4$ and $b=Y$ 
gives $h_{k-1}(Y/2) \ge (Y/4)^2 m/4 = (1152(x-k+1))^{-1} Y$.
By concavity, for $z \in [600 cY, (1-600c)Y]$
we have $h_{k-1}(z) \ge 1200c h_{k-1}(Y/2) 
> \tfrac{c}{x} \tbinom{x}{k-1} >  h_{k-1}(Y-y)$,
so $y \notin [600 cY, (1-600c)Y]$.

For the `furthermore' statement, 
we can assume $x < (1+\gG)k$, with $\gG := e^{-9}$,
as otherwise $Y = \tfrac{x-1}{x-k} \tbinom{x-2}{k-1}
\le \tfrac{1+\gG}{\gG} \tbinom{x-2}{k-1}$,
so $600cY < 10^7 c \tbinom{x-2}{k-1}$.
Let $E = \tfrac{1}{2} \tbinom{x-2}{k-1}$
and define $\xi$ by $X + E = \tbinom{x-1 +\xi}{k}$,
so $0 < \xi \le \tfrac{1}{2}$ by (\ref{mvt}).
We claim that $h'_k(X+E) \ge \tfrac{x}{12(x-k+1/2)(x-k+1)}$.

First we assume the claim and complete the proof.
We have $h'_k(X+z) \ge h'_k(X+E)$ for $z \in [0,E]$,
so $h_k(E) \ge \tfrac{xE}{12(x-k+1/2)(x-k+1)}
\ge \tfrac{1}{18x} \tbinom{x}{k-1}$. Then by concavity
$h_k(X+z) > \tfrac{c}{x} \tbinom{x}{k-1} > h_k(X+y)$
for all $z \in [10^7 c \tbinom{x-2}{k-1}, Y/2]$,
so $y \notin [10^7 c \tbinom{x-2}{k-1}, (1-600c)Y]$.

To prove the claim, we note that $e_k$ has gradient
$\big ( \tbinom {x}{k-1} - \tbinom {x-1}{k-1} \big ) 
\big ( \tbinom {x}{k} - \tbinom {x-1}{k} \big )^{-1} 
\le \tfrac{k}{x-k+1}$, so by Lemma \ref{deriv} 
\begin{align}
\label{equation: deriv lower bound}
 f_k'(X + E) - e_k'(X+z) 
& \ge \frac {k}{x-k+\xi}
 \Big ( 1 - \frac{1}{(x-k+\xi)g_k(x-1+\xi)} \Big ) 
 - \frac {k}{x-k+1} \nonumber \\ 
& = \frac {k(1-\xi) }{(x-k+\xi)(x-k+1)} 
 - \frac {k}{(x-k+\xi )^2g_k(x-1 + \xi)} \nonumber \\ 
& \geq	\frac {1}{x-k+\xi} \Big ( \frac{x}{3(x-k+1)} 
 - \frac {k}{(x-k+\xi)g_k(x-1 + \xi)} \Big ),
\end{align}
as $k(1-\xi) \ge \tfrac{x}{2(1+\gG)} \ge \tfrac{x}{3}$.
As $x \geq k+1$ we have $x-k+\xi \geq (x-k+1)/2$ and $x \leq (1+\gG )k$ 
gives $\log \big ( \tfrac {x-1+\xi }{x-k+\xi } \big ) 
\geq \log \big ( \tfrac {1+ \gG }{\gG } \big ) \geq 8$. Thus
\[ x(x-k+\xi )g_k(x-1+\xi)  \geq x(x-k+\xi ) \log 
  \Big ( \tfrac {x-1 + \xi }{x-k+\xi } \Big )  
 \geq 4k (x-k+1). \]
In combination with \eqref{equation: deriv lower bound} 
this proves the claim, and so the lemma.
\end{proof}

We conclude this section with a technical lemma 
needed in the next section.

\begin{lem}\label{phi}
Let $k \ge 3$. Define $\phi:[1,k+1] \to \mb{R}$ by
$\phi(t) = k - \tfrac {t-1}{2} - \tfrac{k}{x-k+1}$,
where $k \le x \le k+1$ with $\tbinom{x}{k} = t$.
Then $\phi(1)=\phi(k+1)=0$, $\phi$ is concave,
and $\phi(2) > \tfrac{3}{4}$.
\end{lem}

\begin{proof}
We have $\phi(1) = k - \tfrac {1-1}{2}  - \tfrac{k}{k-k+1} = 0$
and $\phi(k+1) = k - \tfrac {k+1-1}{2} - \tfrac{k}{k+1-k+1}=0$.
Also, $t(x) = \tbinom{x}{k}$ is a convex function of $x$,
so has a concave inverse $x(t)$, so $-1/x(t)$ is concave,
so $\phi$ is concave. To estimate $\phi(2)$,
we let $\tT \in (0,1)$ be such that $\tbinom{k+\tT}{k}=2$,
and apply the Mean Value Theorem to get
$2 = \tbinom{k+\tT}{k} \le \tT \tbinom{k+\tT}{k} g_k(k+\tT) 
\le 2\tT \log \tfrac{k+1+\tT}{1+\tT} \le 2\tT \log (k+1)$,
so $\tT \ge 1/\log(k+1)$. Then $\phi(2)
\ge k - \tfrac {1}{2} - \tfrac{k}{1+1/\log(k+1)}
= \tfrac{k}{1+\log(k+1)} - \tfrac {1}{2}
\ge \tfrac{3}{1+\log(4)} - \tfrac {1}{2} > \tfrac{3}{4}$.
\end{proof}

\section{Stability for the Kruskal--Katona theorem}

In this section we prove Theorem \ref{thm: KK-stab}.
We start by recording some basic properties
of shadows that will be used throughout the paper.

\begin{lem}\label{kkprops}
Let $s,k \in \mb{N}$ with $s \ge k$, 
$m=\tbinom{s}{k}$, $m'=\tbinom{s-1}{k}$ 
and $0 \le E_1,E_2 \le \tbinom{s-1}{k-1}$. Then
\begin{enumerate}[(i)]
\item $\pl(\I^{(k)}_{m' + E_1})
= \tbinom{[s-1]}{k-1} \cup ((\pl \I^{(k-1)}_{E_1})+s)$.
\item $\pl(\I^{(k)}_{m+E_2} \sm \I^{(k)}_m) = \I^{(k-1)}_{E_2} 
\cup ((\pl \I^{(k-1)}_{E_2})+(s+1))$.
\item $\pl(\J^{(k)}_{s,E_1,E_2})
= \pl(\I^{(k)}_{m' + E_1})
\cup ((\pl \I^{(k-1)}_{E_2})+(s+1))$.
\item $|\pl(\I^{(k)}_{a+b})| \le |\pl(\I^{(k)}_a)| + |\pl(\I^{(k)}_b)|$,
with strict inequality if $k \ge 2$ and $a \ge b >0$. 
\end{enumerate}
\end{lem}

\begin{proof}
Statements (i) and (ii) are clear, and imply (iii), 
recalling from  (\ref{Jk}) the definition 
of $\J^{(k)}_{s,E_1,E_2}$ and noting that
$\I^{(k-1)}_{E_2} \sub \tbinom{[s-1]}{k-1}$.
For (iv), let $\mc{A}$ be the union of copies of
$\I^{(k)}_a$ and $\I^{(k)}_b$ on disjoint vertex sets.
Then $|\pl(\I^{(k)}_a)| + |\pl(\I^{(k)}_b)| = |\pl(\A)|
\ge |\pl(\I^{(k)}_{a+b})|$ by Kruskal--Katona.
If equality holds then the vertex sets satisfy 
$|V(\I^{(k)}_{a+b})| = |V(\I^{(k)}_a)| + |V(\I^{(k)}_b)|$
by \cite[Corollary 2.2]{Furedi-Griggs}. 
However, this is impossible for $k \ge 2$ and $a \ge b > 0$.
To see this, consider $\mc{A}'$ obtained from $\mc{A}$
by deleting some $v \in V(\I^{(k)}_b)$, say of degree $d$,
and adding $d$ sets $A \cup \{u\}$ with 
$A \in \tbinom{V(\I^{(k)}_a)}{k-1}$ 
and $u \in V(\I^{(k)}_b) \sm \{v\}$.
Then $|\mc{A}'|=|\mc{A}|$ and
$|V(\mc{A})| > |V(\mc{A}')| \ge |V(\I^{(k)}_{a+b})|$.
\end{proof}

Next we show local stability, i.e.\ a sharp estimate for 
the shadow of families that are close to a clique.

\begin{lem}
\label{lem: local stability KK}
Let $\A \subset \tbinom {[n]}{k}$, $s \in [n]$,
$\A_1 = \A \cap \tbinom {[s]}{k}$ and $\A_2 = \A \sm \A_1$. 
Suppose $|\A_1| = \tbinom{s-1}{k} + E_1$ and 
$|\A_2| = E_2$, with $0 \le E_1, E_2 \leq \tbinom{s-1}{k-1}$. 
Then $|\pl (\A)| \geq |\pl(\J^{(k)}_{s,E_1,E_2})|$. 
\end{lem}

\begin{proof}
For $s<t \le n$ let $\A^t_2$ be the set of all 
$A-t \in \tbinom {[n]}{k-1}$ 
where $A \in \A^t_2$ with $\max A = t$. Then
\[ |\pl(\A)| \ge |\pl(\A_1)| + \sum_{t>s} |\pl(\A^t_2)|
\ge |\pl(\I^{(k)}_{|\A_1|})| + \sum_{t>s} |\pl(\I^{(k-1)}_{|\A^t_2|})|
\ge |\pl(\I^{(k)}_{|\A_1|})| + |\pl(\I^{(k-1)}_{|\A_2|})|
= |\pl(\J^{(k)}_{s,E_1,E_2})|, \]
using Kruskal--Katona, then Lemma \ref{kkprops}.iv, 
and finally Lemma \ref{kkprops}.iii.
\end{proof}

Now we describe the compression operations 
that will be used throughout the paper.
Given disjoint sets $U, V \subset [n]$, the $C_{U,V}$ 
compression of a set $A \subset [n]$ is given by 
\begin{equation*}
C_{U,V}(A) := \begin{cases} 
(A \sm U) \cup V \quad &\mbox{if }
U \subset A \mbox { and } V \cap A = \es ;\\ 
A \quad &\mbox{ otherwise}. 	
\end{cases} \end{equation*}
Given a family $\A \subset \{0,1\}^n$ the 
$C_{U,V}$ compression of $\A$, denoted $C_{U,V}(\A)$, is given by 
\begin{equation*}
C_{U,V}(\A) := 
\big \{C_{U,V}(A): A \in \A \big \} 
\cup \big \{A: C_{U,V}(A) \in \A \big \}.
\end{equation*}
The following result, essentially due to Daykin \cite{Daykin} 
(see also \cite{Bol,BL,FF}) shows that 
for any $\A \subset \tbinom {[n]}{k}$ there is a sequence 
of $(U,V)$-compressions which compress $\A$ to an initial 
segment of colex with the property that successive 
compressions do not increase the shadow 
(in particular this proves Kruskal--Katona).

\begin{thm} \label{thm: KK compress}
Let $\A \subset \tbinom{[n]}{k}$ with $|\A| = 	m$. 
Then there is a sequence $\{(U_i,V_i)\}_{i\in [L]}$ 
where $U_i ,V_i \subset [n]$ are disjoint with 
$|U_i| = |V_i|$ for all $i\in [L]$,
such that defining $\A_0 = \A$ and iteratively 
$\A_i := C_{U_i,V_i}(\A_{i-1})$ for $i\in [L]$,
each $|\pl (\A_i)| \leq |\pl (\A_{i-1})|$
and $\A_L = \I^{(k)}_{m}$.
\end{thm} 

As discussed in the introduction, our proof of
Theorem \ref{thm: KK-stab} analyses the reversal 
of the above compressions. To do so, 
in each decompression step in which we might in theory
lose control on the distance from a clique, 
we will apply the following lemma 
which shows that this control is in fact maintained.

\begin{lem} 
\label{lem: KK clique size}
Given $k \in \mb{N}$, $\dD \in (0,1)$ and $c = 10^{-8} \dD$,
if $\A \subset \tbinom{[n]}{k}$ with $|\A| = \tbinom {x}{k}$ and  
$|\pl(\A)| \leq (1 + \tfrac{c}{x}) \tbinom {x}{k-1}$ then
\begin{enumerate}[(i)]
\item $\big | |\A| - \tbinom {M}{k} \big | 
\leq \tfrac{\dD}{2} \tbinom {M-1}{k-1}$ for some 
$M \in \{\lfloor x\rfloor , \lceil x\rceil \}$,
\item if $|\tbinom {S}{k} \sm \A| 
\leq (1-\dD)\tbinom{M-1}{k-1}$ with $|S| = M$ as (i)
then $|\tbinom {S}{k} \sm \A| \leq \dD \tbinom {M-1}{k-1}$.
\end{enumerate}
\end{lem} 
		
\begin{proof} 
The case $k=1$ is trivial. Next we consider $k = 2$. Let $M = |\pl(\A)|$. 
Then $x \le M \le (1 + \tfrac {c}{x}) \tbinom {x}{1} = x + c$
and $|\A| = \tbinom {M\pm c}{2} = \tbinom {M}{2} \pm \dD (M-1)$,
so (i) holds. For (ii), let $\A' = \A \cap \tbinom {S}{k}$,
and note that $|\A'| > \tbinom{M-1}{2}$. Then $\pl \A' = S = \pl \A$
by Kruskal--Katona, so $|\tbinom {S}{k} \sm \A|=0$.
Thus we may assume $k\geq 3$. 
We write $|\A| = \tbinom {x}{k} = X+Y$ with 
$X = \tbinom {x-1}{k}$ and $Y = \tbinom {x-1}{k-1}$.

Next we assume (i) holds and prove (ii).
Write $\A_1 = \A \cap \tbinom {S}{k}$, $\A_2 = \A \sm \A_1$,
$|\A_1| = \tbinom{M-1}{k} + E_1$ and $|\A_2| = E_2$.
We have $0 \le E_1 \le \tbinom {M}{k} - \tbinom {M-1}{k} 
= \tbinom{M-1}{k-1}$
and $0 \le E_2 \le |\tbinom {S}{k} \sm \A| 
+ \tfrac{\dD}{2} \tbinom {M-1}{k-1} 
\le (1-\tfrac{\dD}{2}) \tbinom{M-1}{k-1}$,
so $|\pl (\A)| \geq |\pl(\J^{(k)}_{M,E_1,E_2})|$
by Lemma \ref{lem: local stability KK}.
By the Lov\'asz version of the Kruskal--Katona theorem 
and Lemma \ref{kkprops}.iii we deduce
$f_k(\tbinom{M-1}{k-1}+E_1) + f_{k-1}(E_2) 
< (1 + \tfrac{c}{x}) \tbinom {x}{k-1}$.
With notation as in Lemma \ref{defectapp2},
writing $E_2=Y-y$ we have $\tbinom{M-1}{k-1}+E_1=X+y$,
so $y \notin [600 cY, (1-600c)Y]$.
As $\tbinom{x}{k} = |\A| = \tbinom {M}{k} \pm
\tfrac{\dD}{2} \tbinom {M-1}{k-1}
= (1 \pm \tfrac{\dD}{2}) \tbinom {M}{k}$,
by Lemma \ref{shiftapprox} we have $\tbinom{x-1}{k-1}
= (1 \pm 2\dD) \tbinom {M-1}{k-1}$.
Then $y = Y - E_2 \ge \tbinom{x-1}{k-1} 
-  (1-\tfrac{\dD}{2}) \tbinom{M-1}{k-1} > 600 cY$,
so $y > (1-600c)Y$, giving $E_2 < 600cY$,
and so $|\tbinom {S}{k} \sm \A| \leq 
E_2 + \tfrac{\dD}{2} \tbinom {M-1}{k-1}
< \dD \tbinom {M-1}{k-1}$, as required.

It remains to prove (i) for $k \ge 3$.
We now consider the case $x<k+1$,
so $1 \leq m := |\A| \leq k+1$.
We can assume $m \ge 2$, or (i) holds with $M=k$.
Note that $\I^{(k)}_m = \{ [k+1] \sm \{i\}: i \in [m] \}$
and $\pl(\I^{(k)}_m) = \{ [k+1] \sm \{i,j\}: 
\{i,j\} \cap [m] \ne \es \}$, so by Kruskal--Katona,
$|\pl(\A)| \geq |\pl(\I^{(k)}_{m})| = mk - \tbinom {m}{2}$.
By hypothesis,
$|\pl(\A)| \leq (1 + \tfrac{c}{x}) \tbinom {x}{k-1}$,
where $\tbinom {x}{k-1} = \tfrac{km}{x-k+1}$,
so $m\phi(m) \le \tfrac{c}{x} \tbinom {x}{k-1}$, 
with $\phi$ as in Lemma \ref{phi}.
By concavity, for $t \in [2,k+1-\dD k]$ we have
$t\phi(t) \ge t \phi(2) \tfrac{\dD k}{k-1-\dD k} 
\ge \tfrac{3\dD}{4} \tbinom{x}{k}
\ge \tfrac{3\dD}{4x} \tbinom{x}{k-1}
> \tfrac{c}{x} \tbinom {x}{k-1}$.
Thus $m > k+1-\dD \tbinom{k}{k-1}$, 
i.e.\ (i) holds with $M=k+1$ in this case.

Continuing with the proof of (i),
we can assume $k \ge 3$ and $x \ge k+1$.
Let $M_0 = \lfloor x \rfloor$ and 
$y = \tbinom {M_0}{k} - X$,
so $|\A| = (X+y)+(Y-y)=\tbinom {M_0}{k}+(Y-y)$.
We can assume $y \ge 1$, otherwise (i) holds.
By Kruskal--Katona, Lemma \ref{kkprops}.i and the
Lov\'asz form of Kruskal--Katona we have 
$|\pl A| \ge |\pl(\I^{(k)}_{\tbinom {M_0}{k}+Y-y})|
= \tbinom {M_0}{k-1} + |\pl \I^{(k-1)}_{Y-y}|
\ge f_k(X+y) + f_{k-1}(Y-y)$. By hypothesis,
$|\pl(\A)| \leq (1 + \tfrac{c}{x}) \tbinom {x}{k-1}$,
so Lemma \ref{defectapp2} gives
$y \notin [10^7 c \tbinom{x-2}{k-1}, (1-600c) \tbinom {x-1}{k-1}]$.

Consider the case $y>(1-600c) \tbinom {x-1}{k-1}$.
We have $\tbinom {M_0}{k} = X+y = \tbinom{x}{k}
\pm 600c \tbinom{x-1}{k-1}
= (1 \pm 600c) \tbinom{x}{k}$,
so $\tbinom {M_0-1}{k-1} = (1 \pm 2400c) \tbinom{x-1}{k-1}$
by Lemma \ref{shiftapprox}, so with $M=M_0$ we have
\[ \tbinom {M}{k} = \tbinom{x}{k}
\pm 600c (1 \pm 2400c)^{-1} \tbinom{M-1}{k-1}
= \tbinom{x}{k} \pm 10^4 c \tbinom{M-1}{k-1}.\]
It remains to consider $y < 10^7 c \tbinom{x-2}{k-1}$.
Then $\tbinom {M_0}{k} = X+y = \tbinom{x-1}{k}
\pm 10^7 c \tbinom{x-2}{k-1}
= (1 \pm 10^7 c) \tbinom{x-1}{k}$, as $x \ge k+1$.
By Lemma \ref{shiftapprox} we have
$\tbinom {M_0-1}{k-1} = (1 \pm 10^7 c) \tbinom{x-2}{k-1}$,
so $\tbinom {M_0}{k-1} = \tbinom {M_0}{k} - \tbinom {M_0-1}{k-1}
= \tbinom{x-1}{k} \pm 10^7 c \tbinom{x-2}{k-1}
- (1 \pm 10^7 c) \tbinom{x-2}{k-1}
= \tbinom{x-1}{k-1} \pm 2 \cdot 10^7 c \tbinom{x-2}{k-1}$.
Taking $M = M_0+1$ we have
$\tbinom {M}{k} = \tbinom {M_0}{k} + \tbinom {M_0}{k-1}
= \tbinom{x-1}{k} \pm 10^7 c \tbinom{x-2}{k-1}
+  (1 \pm 2 \cdot 10^7 c) \tbinom{x-1}{k-1}
= \tbinom{x}{k} \pm \frac{\dD}{2} \tbinom{M-1}{k-1}$.	
\end{proof}

We conclude this section
by proving our stability result for Kruskal--Katona.
	 
\begin{proof}[Proof of Theorem \ref{thm: KK-stab}]
Suppose $\dD_0 > 0$,
let $\dD = \min (\tfrac{1}{8}, \tfrac {\dD_0}{3})$
and $c = 10^{-8} \dD$.
Suppose $\A \subset \tbinom {[n]}{k}$ with 
$|\A| = m = \tbinom {x}{k}$ and 
$|\pl(\A)| \leq (1 + \tfrac {c}{x}) \tbinom {x}{k-1}$.
We can assume $m \ge 1$, so $x \ge k$.
By Lemma \ref{lem: KK clique size}.i there is
$M\in \{\lfloor x\rfloor,\lceil x\rceil \}$ with
$\big ||\A| - \tbinom {M}{k} \big | 
\leq \tfrac{\dD}{2} \tbinom {M-1}{k-1}$.
Let $\{(U_i,V_i)\}_{i\in [L]}$ be the sequence of compressions
provided by Theorem \ref{thm: KK compress},
so that $\A_{L} = \I^{(k)}_{m}$ and each
$|\pl(\A_i)| \leq |\pl(\A)| 
\leq ( 1+ \tfrac{c}{x} ) \tbinom {x}{k-1}$.
We show by induction on $i$ with $L \ge i \ge 0$
that there is some $S_i \in \tbinom{[n]}{M}$ with
$|\tbinom {S_i}{k} \sm \A_i| \leq \dD  \tbinom {M-1}{k-1}$.
As $\A_0=\A$, this will prove the theorem,
as we obtain $\big |\tbinom {S_0}{k} \triangle \A \big | 
\leq 3\dD \tbinom {M-1}{k-1} \leq \dD_0 \tbinom {M-1}{k-1}$,
and the `furthermore' statement holds
by Lemma \ref{lem: local stability KK}.

As $\A_{L} = \I^{(k)}_{m}$ the base case 
holds with $S_L = [M]$. For the induction step,
we suppose the required statement for $i$
and prove it for $i-1$. Let 
$\B_j = \tbinom{S_i}{k} \sm \A_j$ for $j \in \{i-1, i\}$. 
The induction hypothesis is
$|\B_i| \leq \dD  \tbinom {M-1}{k-1}$.
Note that if $A \in \A_i \cap \tbinom {S_i}{k}$ 
and $A \notin \A_{i-1} \cap \tbinom {S_{i}}{k}$ 
then $V_i \subset A \subset S_i$ and so 
$|\B_{i-1}| \leq |\B_{i}| + \tbinom {M-|V_i|}{k-|V_i|}$.
In the case that	
$\tbinom {M-|V_i|}{k-|V_i|} < (1-2\dD )\tbinom {M-1}{k-1}$ 
this implies
$|\B_{i-1}| \leq |\B_{i}| + \tbinom {M-|V_i|}{k-|V_i|} 
< (1 - \dD ) \tbinom {M-1}{k-1}$.
As $|\pl(\A_{i-1})| \leq ( 1+ \tfrac{c}{x} ) \tbinom {x}{k-1}$,
Lemma \ref{lem: KK clique size}.ii improves this bound
to $|\B_{i-1}| \le \dD \tbinom {M-1}{k-1}$,
so the induction step holds with $S_{i-1}=S_i$.

It remains to consider the case that
$\tbinom {M-|V_i|}{k-|V_i|} \geq (1-2\dD ) \tbinom {M-1}{k-1}$. 
If $|V_i| \geq 2$ this implies $2\dD (M-1) \geq M-k$.
We may assume that $V_i \subset S_i$ and $U_i \not \subset  S_i $ 
as otherwise $|\B_{i-1}| \leq |\B_i| \leq \dD \tbinom {M-1}{k-1}$. 
Let $T_0 = S_i$ and $T_1 = (S_i\sm V_i )\cup U_i$. We have 
$| \tbinom {S_i \sm V_i}{k}  \sm \A_i| \leq |\tbinom {S_i}{k} \sm \A_i|
 \leq \dD  \tbinom {M-1}{k-1}$ and 
\[ | \tbinom {T_1}{k} \sm \tbinom {S_i \sm V_i}{k} | 
 \leq \tbinom {M}{k} - \tbinom {M - |V_i|}{k} 
= \sum _{i = 1}^{|V_i|} \tbinom {M-i}{k-1} 
    \leq \sum _{i=1}^{|V_i|} 
\Big ( \tfrac {M-k}{M-1} \Big ) ^{i-1} \tbinom {M-1}{k-1}
< \big ( 1 + \tfrac {2\dD }{1-2\dD } \big ) \tbinom {M-1}{k-1},\]
using $\tfrac {M-k}{M-1} \leq 2\dD$ if $|V_i| \geq 2$. 
As $\dD  \leq 1/8$, this gives
\begin{align*}
\Big  |\tbinom {T_0}{k} \sm \A_{i-1} \Big | + 
\Big |\tbinom {T_1}{k} \sm \A_{i-1} \Big | &\leq 
\Big |\tbinom {T_0}{k} \sm \A_{i} \Big | 
+ \Big | \tbinom {S_i \sm V_i}{k}    \sm \A_i \Big | + 
\Big | \tbinom {T_1}{k} \sm \tbinom {S_i \sm V_i}{k} \Big | \\
& < 
\Big ( 1 + 2 \dD + \tfrac {2 \dD }{1 - 2\dD } \Big ) \tbinom {M-1}{k-1} 
< 2(1 - \dD ) \tbinom {M-1}{k-1}.
\end{align*} 
Therefore $|\tbinom {T_j}{k} \sm \A_{i-1}| < (1-\dD ) \tbinom {M-1}{k-1}$
for some $j\in \{0,1\}$. As before,
Lemma \ref{lem: KK clique size}.ii improves this to
$|\tbinom {T_j}{k} \sm \A_{i-1}| \leq \dD \tbinom {M-1}{k-1}$,
so the inductive step is complete with $S_{i-1} := T_j$. 
\end{proof} 

\section{Stability for the cube vertex isoperimetric inequality}
		
In this section we will prove Theorems \ref{thm: ball-sized sets} 
and \ref{thm: VIP stability}. Similarly to our stability result
for Kruskal--Katona, the proofs proceed by analyzing compression operators 
via local stability. We require the existence of a sequence of compressions
that can transform any family $\mc{A}$ into some $\mc{C}$ that is `ball-like',
meaning that $\tbinom {[n]}{\geq k+1} \sub \mc{C} \sub \tbinom {[n]}{\geq k}$
for some $k$. Similarly to before, we require these compressions to maintain 
the size of the family and not increase the size of its vertex boundary.
We also require some further structural properties of the sequence:
we always use compressions $C_{U,V}$ with $|V|=|U|+1$,
and after some initial set of compressions $C_{\es,\{i\}}$ 
the family $\mc{A}_i$ is always an upset, 
i.e.\ if $A \in \mc{A}_i$ and $A \sub B$ then $B \in \mc{A}_i$.
The formal statement is as follows.

\begin{thm}\label{thm: compression theorem}
Given $\A \subset \{0,1\}^n$ there are 
$L_0, L_{1} \in	{\mathbb N}$ with $0\leq L_0 \leq L_1$ 
and pairs of sets $\{(U_i, V_i)\}_{i\in [L_1]}$ so that, 
setting $\A_0 = \A$ and $\A_{i} := C_{U_{i},V_{i}}(\A_{i-1})$ 
for all $i\in [L_1]$, the following hold:
\begin{enumerate}[(i)]
\item $U_i \cap V_i = \es $ for all $i\in [L_1]$;
\item $|V_i| = |U_i| + 1$ for all $i\in [L_1]$;
\item $|U_i| = 0$ for $i\in [L_0]$ and 
$|U_i| \geq 1$ for $i\in [L_0+1,L_1]$;	
\item $\A_i$ is an upset for all $i\in [L_0,L_1]$,				
\item $|\pl_v({\cal A }_{i})| \leq |\pl_v(\A_{i-1})|$ for all $i\in [L_1]$;
\item $\A_{L_1} = \tbinom {[n]}{\geq k+1} \cup \B$ 
where $\B \subset \tbinom {[n]}{k}$ for some $k$.
\end{enumerate}
\end{thm} 
	
It seems that Theorem \ref{thm: compression theorem}
does not appear in the literature, although it is an 
easy extension of known results (similar  
statements are given in \cite{Bol, BL, Daykin, FF}),
so rather than giving a complete proof we will just
briefly indicate why the required sequence of compressions exists: 
\begin{itemize}
\item Given $\A \subset \{0,1\}^n$, 
the family $C_{\es , \{i\}}({\A })$ has the 
same size as $\A$ and has vertex boundary at most that of $\A$.
Repeatedly applying such compressions for different $i\in [n]$, 
we obtain an upset with vertex boundary at most that of $\A$. 
\item Given disjoint sets $U,V \subset [n]$ with $|U| < |V|$, 
the family $C_{U,V}(\A)$ has at least as many elements
of $\tbinom {[n]}{\geq k}$ as 	$\A$. Furthermore, if $\A$ 
is not ball-like then there are disjoint sets $U,V \subset [n]$ 
with $|V| = |U| + 1$ so that $C_{U,V}(\A)$ 
is closer to a ball-like set.
\item If $C_{U',V'}(\A) = \A$ for all 
$U' \subset U$ with $|U'| = |U|-1$ and 
$V' \subset V$ with $|V'| = |V|-1$ then
$|\pl_v({C}_{U,V}(\A))| \leq |\pl_v(\A)|$
and ${C}_{U,V}(\A)$ is closer to a ball-like set.
Furthermore, if $\A$ is an upset then so is $C_{U,V}(\A)$.
\end{itemize}
From the above facts, Theorem \ref{thm: compression theorem}
follows by repeatedly applying compressions $C_{U,V}$ to $\A$ 
where $|V|=|U|+1$ is minimal with $C_{U,V}(\A) \neq \A$.
The proofs of Theorems \ref{thm: ball-sized sets}
and \ref{thm: VIP stability} will analyze the reversal of these compressions. 
In the next two subsections we will prove a 
local stability version of Harper's Theorem
and collect various estimates that boost the accuracy
of approximation by a generalised Hamming ball
for a family with small vertex boundary. 
In the third subsection we prove a stability theorem for 
families of size close to a ball, which implies Theorem 
\ref{thm: ball-sized sets}. The main result in the fourth
subsection allows us to reverse the compressions from 
Theorem \ref{thm: compression theorem} for $i \geq L_0$. 
In particular, this will show that upsets with small vertex boundary
are close to generalised Hamming balls of the first type.
The second type of generalised Hamming ball then appears under reversal 
of the compressions for $i \in [0, L_0-1]$;
the analysis of these steps is given in the fifth subsection,
using the local stability theorem and 
the stability theorem for ball-sized sets.
The final subsection contains the proof of Theorem \ref{thm: VIP stability}.		

\subsection{Local stability for the vertex isoperimetric inequality} 

The main result of this subsection is our local stability result
for perturbations of a generalised Hamming ball. Recall that 
$\J_{m,D,E} = \I_{m-D} \cup (\I_{m+E} \sm \I_{m})$.
For $\F \sub \{0,1\}^n$ and $i \ge 0$ we define the
{\em iterated neighbourhoods} $N^i(\F)$ by $N^0(\F)=\F$
and $N^{i+1}(\F)=N^i(\F) \cup \pl_v(N^i(\F))$.
We start with some identities for the vertex boundary
and iterated neighbourhoods of $\J_{m,D,E}$.

\begin{lem} \label{lem:plJ}
Let $n \ge t \ge k \ge i$, $0 \le D,E \le \tbinom{t-1}{k-1}$
and $m = \tbinom{[n]}{\ge k+1} + \tbinom{t}{k}$.
Then $|N^i(\J_{m,D,E})| + |N^i(\I_m)|
= |N^i(\I_{m-D})| + |N^i(\I_{m+E})|$,
so $|\pl_v(\J_{m,D,E})| + |\pl_v(\I_m)|
= |\pl_v(\I_{m-D})| + |\pl_v(\I_{m+E})|$.
\end{lem}

\begin{proof} 
The statement on vertex boundaries 
is equivalent to that on neighbourhoods with $i=1$.
Writing $T=\tbinom{t}{k}$, we have
\begin{align*}
|N^i(\I_{m-D})| & = \tbinom{n}{\ge k+1-i} 
+ |\pl^i(\I^{(k)}_{T-D})|, \\
|N^i(\I_{m+E})| & = \tbinom{n}{\ge k+1-i} 
+ \tbinom{t}{k-i} + |\pl^i(\I^{(k-1)}_E)|, \\
|N^i(\I_m)| & = \tbinom{n}{\ge k+1-i} + \tbinom{t}{k-i}, 
\text{ and } \\
|N^i(\J_{m,D,E})| & = \tbinom{n}{\ge k+1-i} 
+ |\pl^i(\I^{(k)}_{T-D})| + |\pl^i(\I^{(k-1)}_E)|.
\end{align*}
The lemma follows.
\end{proof} 

Now we prove our local stability result.
The main task of the proof is to establish a submodularity property 
for (iterated) neighbourhoods that may have independent interest.

\begin{lem} \label{lem: local stability}
Suppose $\mc{A}, \mc{G} \sub \{0,1\}^n$.
Let $\A^- = \A \cap \G$ and $\A^+ = \A \cup \G$.
For any $i \ge 0$ we have
$|N^i(\A)| + |N^i(\G)|
\geq |N^i(\A^-)| + |N^i(\A^+)|$,
so $|\pl_v(\A)| + |\pl_v(\G)|
\geq |\pl_v(\A^-)| + |\pl_v(\A^+)|$.

Suppose also $\G$ is a generalised Hamming ball,
namely $\G = \tbinom {[n]}{\geq \ell +1} \cup \tbinom {[t]}{\ell }$ 
with $\ell \leq t \leq n$, or 
$\G = \tbinom {[n]}{\geq \ell +1} 
\cup \tbinom {[t-1]}{\ell } \cup \tbinom {[t-1]}{\ell -1}$ 
with $\ell +1 \leq t \leq n-1$. 
Write $m = \tbinom{n}{\geq \ell +1} + \tbinom {t}{\ell }$,
$|\A^-| = |\G| - D$, $|\A^+| = |\G| + E$ and suppose
$D,E \leq \tbinom {t-1}{\ell -1}$.
Then $|N^i(\A)| \geq |N^i(\J_{m,D,E})|$,
so $|\pl_v(\A)| \geq |\pl_v(\J_{m,D,E})|$.
\end{lem}
		
\begin{proof}
As $|\A| + |\G| = |\A^-| + |\A^+|$,
the statement on vertex boundaries 
is equivalent to that on neighbourhoods with $i=1$.
Let $\E = N^i(\A^+) \sm (N^i(\A) \cup \G)$.
Then $|N^i(\A^+) \sm \G| \le |N^i(\A) \sm \G| + |\E|$, so
\begin{equation}\label{equation: upper boundary}
|N^i(\A) \sm \G| \ge |N^i(\A^+) \sm \G| - |\E|
= |N^i(\A^+)| - |\G| - |\E|,
\end{equation} 
as $\G \subset \A^+$.
Next we observe that $\E \subset N^i(\G) \sm \G$ 
(as $N^i(\A^+) = N^i(\A) \cup N^i(\G)$)
and $N^i(\A^-) \cap \E = \es$
(as $N^i(\A^-) \sub N^i(\A)$), so
$|N^i(\A^-) \cap (N^i(\G) \sm \G)| \le |N^i(\G)| - |\G| - |\E|$.
We deduce 
\begin{equation} \label{equation: lower boundary}
|N^i(\A) \cap \G | \ge |N^i(\A^-) \cap \G |
= |N^i(\A^-)| - |N^i(\A^-) \cap (N^i(\G) \sm \G)|
 \geq |N^i(\A^-)| + |\G| + |\E| - |N^i(\G)|.
\end{equation}
Combining \eqref{equation: upper boundary} with 
\eqref{equation: lower boundary} gives
\begin{equation*}
|N^i(\A)| = |N^i(\A) \cap \G | + |N^i(\A) \sm \G|
\geq |N^i(\A^-)| + |N^i(\A^+)| - |N^i(\G)|,
\end{equation*}
which is the first statement of the lemma. Now
\begin{equation*}
|\pl _v(\A)| \geq |\pl_v(\I_{m-D})| + |\pl_v(\I_{m+E})|  
- |\pl_v(\G)| = |\pl_v(\J_{m,D,E})|,
\end{equation*}
by Harper's Theorem applied to $\A^+$ and $\A^-$ 
and then Lemma \ref{lem:plJ}.
\end{proof}	

We conclude this subsection by showing how the local stability 
obtained in the previous lemma allows us to boost the accuracy
of approximation by a generalised Hamming ball
for a family with small vertex boundary. 

\begin{lem} \label{localstabcor}
Let $\dD \in (0,1)$, $c = 10^{-9}\dD$
and $\A \sub \{0,1\}^n$ with
$|\A| = \tbinom {n}{\geq k+1} + \tbinom {x}{k}$,
and $\tbinom {x}{k} = \tbinom {|S|}{k} 
\pm \tfrac{\dD }{5}\tbinom {|S|-3}{k-2}$, 
where $2 \leq k \leq |S| \leq n-1$. 
Suppose $|\pl_v({\A})| \leq \Lov {\A} 
+ \tfrac {ck(x-k)}{x^3} \tbinom {x}{k-1}$ and 
$|\A \sm \G| \leq \tbinom {|S|-1}{k-1} - \dD \tbinom {|S|-3}{k-2}$ 
for some generalised Hamming ball $\G$ 
with $|\G| = \tbinom {n}{\geq k+1} + \tbinom {|S|}{k}$.
Then $|\A \triangle \G| \leq \dD \tbinom {|S|-3}{k-2}$.
\end{lem}

\begin{proof}
We apply Lemma \ref{lem: local stability} 
to $\A$ and $\G$, which gives
$|\pl_v(\A)| \geq |\pl_v(\J_{m,D,E})|$, where
$m = \tbinom {n}{\geq k+1}+\tbinom {|S|}{k}$,
$D = |\G \sm \A|$ and $E = |\A \sm \G|$. 
Note that $E \le \tbinom {|S|-1}{k-1} 
- \dD \tbinom {|S|-3}{k-2}$
and $|\pl_v(\J_{m,D,E})| 
= \tbinom {n}{k} - \tbinom {x}{k} 
+ |\pl(\J^{(k)}_{|S|, E',E})|$,
where $E' = \tbinom {|S|-1}{k-1} - D$. 
Our assumed upper bound on $|\pl_v(\A)|$ 
implies $|\pl(\J^{(k)}_{|S|, E',E})| \leq 
\big ( 1 + \tfrac {ck(x-k)}{x^3} \big ) \tbinom {x}{k-1}$. 
By Lemma \ref{lem: KK clique size}.ii, applied with 
$\tfrac {c k(x-k)}{x^2}$ in place of $c$,
we obtain $D = |\tbinom{S}{k} \sm \J^{(k)}_{|S|, E',E}|
\le \tfrac{\dD}{5} \tbinom {|S|-3}{k-2}$,
and so $|\A \triangle \G| \leq \dD \tbinom {|S|-3}{k-2}$.
\end{proof}

\subsection{Boosting approximations}

In this subsection we collect several further lemmas
for boosting approximations under the assumption of
small vertex boundary. We start by quantifying the 
defect in \eqref{equation: coarse VIP bound}
for families that are somewhat close 
to a generalised Hamming ball.
		
\begin{lem} \label{cleandefect}
Let $n \geq t \geq \ell \geq 2$, and let 
$\G$ be a generalised Hamming 
ball of size $m = \tbinom {n}{\geq \ell+1} + \tbinom {t}{\ell}$. 
Suppose $\A \subset \{0,1\}^n$ 
and $|\A| = \tbinom {n}{\geq \ell+1} + 
\tbinom {t-1}{\ell} + E_1 + E_2$ with $|\A\sm \G| = E_2$, 
where $1 \leq E_1, E_2 \leq \tbinom {t-1}{\ell-1}$. 
Set $E_{min} := \max ( 0 , E_1 + E_2 - \tbinom {t-1}{\ell-1})$ 
and $E_{max} := \min ( E_1 + E_2 , \tbinom {t-1}{\ell-1})$. 
Then $$|\pl_v(\A)| - \Lov {\A } 
\geq \Phi := \big ( f_{\ell-1}(E_1) + f_{\ell-1}(E_2) \big ) 
- \big ( f_{\ell-1}( E_{min} ) + f_{\ell-1}( E_{max} ) \big ).$$
\end{lem}
		
\begin{proof}
Note that $E_1 + E_2 = E_{min} + E_{max}$ 
and $E_{min} \leq E_1, E_2 \leq E_{max}$.
Lemma \ref{lem: local stability} gives   
$|\pl_v(\A)| \geq |\pl_v(\J_{m,D,E_2})|$ with 
$D = \tbinom {t-1}{\ell-1} - E_1$. 
Writing $m' = \tbinom {n}{\geq \ell+1} 
+ \tbinom {t-1}{\ell}$, we have 
$\J_{m,D,E_2} = \I_{m' +E_1} \cup 
\big ( \I_{m + E_2} \sm \I_m \big )$ and  
\begin{align*}
|\pl_v(\J_{m, D,E_2})| 
& = |\pl_v(\I_{m'})| + |\pl(\I^{(\ell-1)}_{E_1})| 
+ |\pl(\I^{(\ell-1)}_{E_2})| - \big ( E_1 + E_2 \big )\\ 
& \geq 	|\pl_v(\I_{m'})| + f_{\ell-1}(E_1) 
+ f_{\ell-1}(E_2) - \big ( E_{1} + E_{2} \big )\\
& =  |\pl_v(\I_{m'})| + f_{\ell-1}(E_{min}) 
+ f_{\ell-1}(E_{max}) + \Phi - \big ( E_{min} + E_{max} \big ),
\end{align*}
where the inequality holds by the Lov\'asz form of Kruskal--Katona
applied to $\I^{(\ell-1)}_{E_1}$ and $\I^{(\ell-1)}_{E_2}$. 
As $|\pl_v(\I_{m'})| = 	\tbinom {n}{\ell} - \tbinom {t-1}{\ell} + 
f_{\ell}(\tbinom {t-1}{\ell})$, it remains to show
\[ \Psi := \tbinom {n}{\ell} + f_{\ell}(\tbinom {t-1}{\ell}) 
+ f_{\ell-1}(E_{min}) + f_{\ell-1}(E_{max})  
- \big (\tbinom {t-1}{\ell} +  	E_{min} + E_{max} \big ) 
\geq \Lov {\A }.\] 
We prove this inequality according to the cases
$E_{min} = 0$ or $E_{max} = \tbinom {t-1}{\ell-1}$
(one of which must hold).

First consider $E_{min} = 0$.
Then $E_{max}=E_1+E_2 \le \tbinom {t-1}{\ell-1}$.
Define $x \ge \ell$ by 
$\tbinom {x}{\ell} = |\A| - \tbinom{n}{\ge \ell+1}
= \tbinom {t-1}{\ell} + E_{max}$ and note that $x \le n$.
By Lemma \ref{concavecor} we have 
$\Psi \geq \tbinom {n}{\ell} 
+ f_{\ell}( \tbinom {x}{\ell} - E_{max} ) + f_{\ell-1}(E_{max}) 
- \tbinom {x}{\ell}
\ge  \tbinom{n}{\ell} - \tbinom {x}{\ell} + \tbinom {x}{\ell-1}
= \Lov {\A }$, as required.

It remains to consider $E_{max} = \tbinom {t-1}{\ell-1}$.
Note that $f_\ell(\tbinom {t-1}{\ell}) + f_{\ell-1}(E_{max})
= \tbinom{t-1}{\ell-1} + \tbinom {t-1}{\ell-2}
= f_\ell(\tbinom {t}{\ell})$, so $\Psi = \tbinom {n}{\ell} + 
f_\ell(\tbinom {t}{\ell}) + f_{\ell-1}(E_{min}) - 
\big ( \tbinom {t}{\ell} + E_{min} \big )$. 
If $t = n$ then $|\A| = \tbinom {n}{\geq \ell} + E_{min}$ and
$\Psi = \tbinom {n}{\ell-1} + f_{\ell-1}(E_{min}) - E_{min} =\Lov \A$. 
If $t < n$ then $|\A| = \tbinom {n}{\geq \ell+1} + \tbinom {x}{\ell}$ 
with $\tbinom {x}{\ell} = \tbinom {t}{\ell} +E_{min} < \tbinom {n}{\ell}$. 
Similarly to the previous case, by Lemma \ref{concavecor} we have 
$\Psi = \tbinom {n}{\ell} - \tbinom {x}{\ell}
+ f_\ell(\tbinom {x}{\ell} - E_{min}) + f_{\ell-1}(E_{min})
\ge \tbinom {n}{\ell} - \tbinom {x}{\ell} +  \tbinom {x}{\ell-1}
= \Lov {\A }$.
\end{proof}

Our next lemma boosts the accuracy of approximation in the `ball part' 
of a family which is not (necessarily) ball-sized.

\begin{lem} \label{vbcorrection1}
Let $\dD \in (0,1)$, $c=10^{-3}\dD$ and $\mc{A} \sub \{0,1\}^n$.
Suppose $|\A| = \tbinom {n}{\geq k+1} + \tbinom {x}{k}$,
with $\tbinom {x}{k} = \tbinom {s}{k} 
\pm \tfrac{\dD }{5}\tbinom {s-3}{k-2}$, where 
$2 \leq k \leq s \leq n-1$ and $s \in \mb{N}$. 
Suppose also $|\A \sm \tbinom {[n]}{\geq k+1}| 
< \tbinom {n-1}{k} - \dD \tbinom {s-3}{k-2}$
and $|\pl_v({\A})| \leq \Lov {\A} 
+ \tfrac {ck(x-k)}{x^3} \tbinom {x}{k-1}$.
Then $|\tbinom {[n]}{\geq k+1} \sm \A| 
\leq \dD \tbinom {s-3}{k-2}$. 
\end{lem}
	
\begin{proof}
Let $E_1 = \tbinom {n-1}{k} - |\tbinom {[n]}{\geq k+1} \sm \A|$ 
and $E_2 = |\A \sm \tbinom {[n]}{\geq k+1}|$.
Then $E_1+E_2 = \tbinom {n-1}{k} + \tbinom{x}{k}$,
$E_{max} = \tbinom {n-1}{k}$ and $E_{min} = \tbinom{x}{k}$.
By assumption $E_2 \le \tbinom{n-1}{k} - \dD \tbinom {s-3}{k-2}$,
so $E_1 \ge \tbinom{x}{k} + \dD \tbinom {s-3}{k-2}$.
We may also assume $E_2 \ge 1$, as otherwise we are done. 
Then Lemma \ref{cleandefect} gives $|\pl_v(\A)| - \Lov {\A } 
\geq \big ( f_{k}(E_1) + f_{k}(E_2) \big ) 
- \big ( f_{k}( \tbinom{x}{k} ) + f_{k}( \tbinom {n-1}{k} ) \big )$.
By Lemma \ref{defectapps}.iii,
applied with $k$ and $\tfrac{ck(x-k)}{x^2}$
in place of $\ell$ and $c$ we have
$\min\{E_1,E_2\} \le \tbinom{x}{k} 
+ 250 c \tbinom{x-3}{k-2}$.
This bound must apply to $E_2$,
and we deduce $|\tbinom {[n]}{\geq k+1} \sm \A|
= E_2 - \tbinom{x}{k} \leq \dD \tbinom {s-3}{k-2}$.
\end{proof}
	
In the next lemma, with a proof similar to the previous one
but somewhat more involved, we boost the accuracy of 
approximation to a ball for sets that are approximately ball-sized.

\begin{lem} \label{vbcorrection2}
Let $\dD \in (0,1)$, $c=10^{-3}\dD$ and $\mc{A} \sub \{0,1\}^n$.
Suppose $|\A| = \tbinom {n}{\geq k} 
\pm \tfrac {\dD }{5} \tbinom {x-1}{k-1}$,
where $2 \le k \le x \le n$.
If $k=2$ suppose also that $|\A| \le \tbinom {n}{\geq 2}$.
Suppose $|\pl_v({\A})| < \Lov {\A} + \tfrac {c}{x} \tbinom {x}{k-1}$ 
and $|\A \sm \tbinom {[n]}{\geq k}| 
< \tbinom {n-1}{k-1} - \dD \tbinom {x-1}{k-1}$.
Then $|\A \triangle \tbinom {[n]}{\geq k}| 
\le \dD \tbinom {x-1}{k-1}$. 
\end{lem}		

\begin{proof}
Let $E_1 = \tbinom {n-1}{k-1} - |\tbinom {[n]}{\geq k} \sm \A|$ and $E_2 = |\A \sm \tbinom {[n]}{\geq k}|$. Note that $E_1+E_2-\tbinom {n-1}{k-1} = |\A \sm \tbinom {[n]}{\geq k}| 
- |\tbinom {[n]}{\geq k} \sm \A| = |\A| - |\tbinom {[n]}{\geq k}|$. The hypotheses give $E_1, E_2 < \tbinom {n-1}{k-1}$ and $E_1 \geq 1$. We may also assume $E_2 \ge 1$, as otherwise we are done. For $k=2$ note that this is already contrary to the hypothesis. Indeed, taking $m = \binom {n}{\geq 2}$ and $D = \binom {n-1}{k-1}-E_1$, Lemma \ref{lem: local stability} gives $|{\partial }_v(A)| \geq |{\partial }_v({\cal J}_{m,D,E})| = 2^{n} - |{\cal A}| \geq B_{lov}({\cal A}) + 1$. Thus in this case $E_2 = 0$ and we are done. 

We now assume $k \geq 3$. Applying Lemma \ref{cleandefect} we obtain $|\pl_v(\A)| - \Lov {\A } 
\geq \Phi := \big ( f_{k-1}(E_1) + f_{k-1}(E_2) \big ) 
- \big ( f_{k-1}( E_{min} ) + f_{k-1}( E_{max} ) \big )$. We will argue according to $|{\cal A}|$. 

First consider the case $|\A| \le \tbinom {n}{\geq k}$. Then $E_{min}=0$ 
and $\tbinom{n-1}{k-1} - \tfrac {\dD }{5} \tbinom {x-1}{k-1}
\le E_{max} = E_1+E_2 \le \tbinom{n-1}{k-1}$.
Also, $\Phi \le \tfrac {c}{x} \tbinom {x}{k-1}$
and $E_1 \ge \frac {4\dD}{5} \tbinom {x-1}{k-1}$.
We have $|\A \triangle \tbinom {[n]}{\geq k}| = D+E_2$
where $D := |\tbinom {[n]}{\geq k} \sm \A| = \tbinom{n-1}{k-1}-E_1
\le E_2 + \tfrac {\dD }{5} \tbinom {x-1}{k-1}$, so it suffices 
to show $E_2 < \tfrac {2\dD}{5} \tbinom {x-1}{k-1}$.
Lemma \ref{defectapps}.i gives $\min\{E_1,E_2\} 
\le 400 c \tbinom{x-1}{k-1} \leq \tfrac {2\dD}{5} \tbinom{x-1}{k-1}$.
This upper bound is less than our lower bound on $E_1$, so applies to $E_2$.


It remains to consider
$|\A| > \tbinom {n}{\geq k}$.
Here we have $E_{max} = \tbinom {n-1}{k-1}$ and
$0 \le E_{min}=E_1+E_2-\tbinom {n-1}{k-1} 
\le \tfrac {\dD }{5} \tbinom {x-1}{k-1}$. Then $|\A \triangle \tbinom {[n]}{\geq k}| = D+E_2 = 2E_2 - E_{min}$. However Lemma \ref{defectapps}.ii gives $E_2 \le E_{min} + 400 c \tbinom{x-1}{k-1} \leq 
\frac {1}{2} E_{min} + \big ( \frac {\delta }{10} + 400 c \big )\binom {x-1}{k-1} \leq \frac {\delta }{2} \binom {x-1}{k-1} + \frac {1}{2} E_{min}$, which rearranging proves 
$|{\cal A}\triangle \binom {[n]}{\geq k}| \leq \delta \binom {x-1}{k-1}$ as required.
\end{proof} 
	
Our final lemma of this subsection
relates the vertex boundary of $\A$
to that of its sections, namely the families
$\A^0$ and $\A^1$ in $\{0,1\}^{n-1}$ defined by
\begin{equation}\label{eqsec}
\A^j = \{x \in \{0,1\}^{n-1}: (x,j) \in \A \}.
\end{equation}
We use superscripts of $(n-1)$
to avoid confusion between $\{0,1\}^{n-1}$ and $\{0,1\}^n$.

\begin{lem} \label{sections}
Let $\dD \in (0,1)$, $c = 10^{-3}\dD$
and $\A \subset \{0,1\}^n$ with 
$|\A| = \tbinom {n}{\geq k+1} + \tbinom {x}{k}$, 
where $\tbinom {x}{k} = \tbinom {s}{k} 
\pm \tfrac {\dD}{5} \tbinom{s-3}{k-2}$ 
for some $s \in [k,n-1]$. 
Suppose $|\pl_v(\A )| \leq \Lov {\A } + \Phi $ and 
$|\A^1|\geq |\A^0| \geq \tbinom {n-1}{\geq k+1} + \tbinom {x-1}{k}$. Then: 
\begin{enumerate}[(i)]
\item  
$|\pl_v^{(n-1)}(\A^0)|  \leq \LovDim {n-1} {\A^0} + \Phi$  and 
$|\pl_v^{(n-1)}(\A^1)| \leq  \LovDim {n-1} {\A^1} + \Phi $.
\item If $k\geq 2$ and $\Phi \leq \tfrac {ck(x-k)}{x^3} \tbinom {x}{k-1}$ 
then $|\A^0| = \tbinom {n-1}{\geq k+1} + 
\tbinom {x-1}{k} \pm  \dD \tbinom {x-1}{k-1}$ 
or $|\A^0| = \tbinom {n-1}{\geq k+1} + 
\tbinom {x}{k} \pm \dD \tbinom {x-1}{k-1}$.
\end{enumerate} 
\end{lem}

\begin{proof}
Write $X = |\pl_v(\A )| - \Lov {\A }$ and
$X_j = |\pl_v^{(n-1)}(\A^j)| - \LovDim{n-1}{\A^j}$.
Then $X$, $X_0$ and $X_1$ are non-negative
by the Lov\'asz form of Harper's theorem.
We will show $X \ge X_0 + X_1$, which implies (i). 
First we note that $|\pl_v({\A })| \geq 
|\pl_v^{(n-1)}({\A }^0)| + |\pl_v^{(n-1)}({\A }^1)|$,
so it suffices to show
$\LovDim {n-1}{\A^0} + \LovDim {n-1}{\A^1} \ge \Lov{\A}$.
We let $E_j = |\A^j| - \tbinom {n-1}{\geq k+1}$ for $j=0,1$
and consider two cases according to the value of $E_1$.

The first case is $E_1 \leq \tbinom {n-1}{k}$.
Note that $\tbinom{x}{k} \le E_0 \le E_1 \leq \tbinom {n-1}{k}$.
We have $\LovDim {n-1}{\A^j} = \tbinom{n-1}{k} - E_j + f_k(E_j)$ 
for $j=0,1$. As $E_0 + E_1 = \tbinom{x}{k} + \tbinom{n-1}{k}$,
by concavity $f_k(E_0)+f_k(E_1) \ge f_k(\tbinom{x}{k}) + f_k(\tbinom{n-1}{k})
= \tbinom{x}{k-1} + \tbinom{n}{k-1}$, so
$\LovDim {n-1}{\A^0} + \LovDim {n-1}{\A^1} 
\ge 2 \tbinom{n-1}{k} - (\tbinom{x}{k} + \tbinom{n-1}{k})
+ (\tbinom{x}{k-1} + \tbinom{n-1}{k-1})
= \tbinom{n-1}{k} - \tbinom{x}{k} + \tbinom{x}{k-1} = \Lov{\A}$,
as required for (i). For (ii), first note that if $s=n-1$
then $\tbinom {n-1}{k} -\tfrac {\dD}{5} \tbinom{n-4}{k-2}
\le E_0 \le \tbinom {n-1}{k}$,
so $E_0 = \tbinom {x}{k} \pm  \dD \tbinom {x-3}{k-2}$ 
by Lemma \ref{shiftapprox}.
If $s \le n-2$ then the previous calculation gives
$\Phi \ge X_0 + X_1 \ge f_k(E_0)+f_k(E_1) - 
\big( f_k(\tbinom{x}{k}) + f_k(\tbinom{n-1}{k}) \big)$,
so $E_0 \le \tbinom{x}{k} + \dD \tbinom {x-3}{k-2}$
by Lemma \ref{defectapps}.iii.
  
The second case is $E_1 \geq \tbinom {n-1}{k}$, say
$E_1 = \tbinom {n-1}{k} + E'_1$ with $E'_1 \ge 0$.
Note that $E_0 + E'_1 = \tbinom{x}{k}$.
By the lemma hypotheses, $E_0 \geq \tbinom {x-1}{k}$,
so $E'_1 \leq \tbinom {x-1}{k-1}$. Adopting the notation
of Lemma \ref{defectapp2}, we write $X = \tbinom {x-1}{k}$,
$Y = \tbinom {x-1}{k-1}$, $E_0=X+y$, $E'_1=Y-y$ with $0 \le y \le Y$.
By concavity we have $f_k(E_0)+f_{k-1}(E'_1)
\ge f_k(X)+f_{k-1}(Y) =  \tbinom{x}{k-1}$. We have
$\LovDim {n-1}{\A^1} = \tbinom{n-1}{k-1} - E'_1 + f_{k-1}(E'_1)$,
so $\LovDim {n-1}{\A^0} + \LovDim {n-1}{\A^1} 
\ge \tbinom{n-1}{k} - E_0 + f_k(E_0)
+ \tbinom{n-1}{k-1} - E'_1 + f_{k-1}(E'_1)
= \tbinom{n}{k} - \tbinom{x}{k} + f_k(E_0) + f_{k-1}(E'_1)
\ge \tbinom{n-1}{k} - \tbinom{x}{k} + \tbinom{x}{k-1} = \Lov{\A}$,
as required for (i). For (ii), the same calculation gives
$f_k(E_0) + f_{k-1}(E'_1) < \tbinom{x}{k-1} + \Phi$. For $k\geq 3$ by Lemma \ref{defectapp2}, applied with $\tfrac {ck(x-k)}{x^2}$ in place of $c$,
we have $E_0 \le \tbinom {x-1}{k} + 600 \tfrac {ck(x-k)}{x^2} \tbinom {x-1}{k-1}
\le \tbinom {x-1}{k} + \dD \tbinom {x-3}{k-2}$. 

It remains to show (ii) when $k=2$ and $E_1 \geq \binom {n-1}{k}$. 
Note that here $\binom {x}{k} = \binom {s}{k} \pm \frac {\delta }{5}$,
so $\binom {x}{k} = \binom {s}{k}$. However, if $E_0 > \binom {s-1}{2}$ and $E_1' >0$ then applying Harper's theorem to both ${\cal A}^0$ and ${\cal A}^1$ gives $|{\partial }_v({\cal A})| \geq (\binom {n-1}{2}- E_0 + s ) + ( \binom {n-1}{1} - E_1' + 1) = \Lov {{\cal A}} + 1 > \Lov {{\cal A}} + \Phi $, which is a contradiction.
Thus either $E_0 = \binom {s-1}{2}$ or $E_1' = 0$, as required. \end{proof}
	
\subsection{Stability for ball-sized sets}

In this subsection we will prove our first stability result 
for the vertex isoperimetric inequality, which applies to 
families with size close to that of a Hamming ball;
the case $|{\cal A}| = \binom {n}{\geq k}$
implies Theorem \ref{thm: ball-sized sets}.
		
\begin{thm} \label{thm: ball-sized stability}
Suppose $\dD \in (0,1/4)$ and $\A \sub \{0,1\}^n$ 
with $|\A| = m \pm \frac {\delta }{5} \binom {n-1}{k-1}$,
where $m = \tbinom{n}{\geq k}$ and 
$|\pl_v(\A)| \leq (1 + \tfrac{c}{n}) \tbinom{n}{k-1}$,
with $c = 10^{-3} \dD$. If $k=2$ suppose 
also that $|\A| \le m$. Then 
$|\A\triangle \B| \leq \dD \tbinom{n-1}{k-1}$ 
for some Hamming ball $\B$.
Furthermore, $|\pl_v(\A)| \geq |\pl_v(\J_{m,D,E})|$
where $D = |\B \sm \A|$ and $E = |\A \sm \B|$.
\end{thm}
	
\begin{proof}
Let $\{(U_i,V_i)\}_{i\in [L_1]}$ be the sequence of compressions
provided by Theorem \ref{thm: compression theorem}.
We show by induction on $L_1 \ge i \ge 0$ that
there is a Hamming ball $\B_i$ of radius $n-k$ 
such that $|\B_i \triangle \A_i| \leq \dD \tbinom {n-1}{k-1}$. 
As $\A_0=\A$ this will prove the theorem (the `furthermore' 
statement following from Lemma \ref{lem: local stability}).
Initially, it holds with $\B_{L_1} = \B := \tbinom {[n]}{\geq k}$,
as $\A_{L_1} = {\cal I}_{|\A |}$ and $|\A| = \tbinom {n}{\geq k} 
\pm \tfrac {\dD }{5} \tbinom {n-1}{k-1}$.

For $L_1 \ge i \ge L_0$ we show the required statement
with $\B_i = \B$. Suppose $i\in [L_0,L_1-1]$ and
$|\tbinom {[n]}{\geq k} \triangle \A_{i+1}| 
\leq \dD \tbinom {n-1}{k-1}$. As $|V_i|=|U_i|+1$,
if $A \in (\B \sm \A_i) \sm (\B \sm \A_{i+1})$
we have $|A|=k$, $V \sub A$ and $U \cap A = \es$.
The number of such sets $A$ is 
$\tbinom{n-|U|-|V|}{k-|V|} \le \tbinom {n-3}{k-2}$,
so $|\B \sm \A_i| \leq |\B \sm \A_{i+1}| + \tbinom {n-3}{k-2} 
\leq (\dD + \tfrac {1}{2} ) \tbinom {n-1}{k-1} 
\leq (1 - \dD ) \tbinom {n-1}{k-1}$ as $\dD < \tfrac {1}{4}$.
Lemma \ref{vbcorrection2} (with $x=n$) improves this to 
$|\B \sm \A_i| \leq \dD \tbinom {n-1}{k-1}$, as required.

Now suppose $i\in [0,L_0-1]$ and 
$|\B_{i+1} \sm \A_{i+1}| \leq \dD \tbinom {n-1}{k-1}$ 
where $\B_{i+1} = \B^n_{n-k}(A_{i+1})$ is a Hamming ball 
of radius $n-k$, centred at $A_{i+1} \subset [n]$. 
We have $U_i = \es $ and $V_i = \{s\}$ for some $s\in [n]$. 
Let $\B^{(1)} = \B_{i+1}$ and 
$\B^{(2)} = \B_{i+1} \triangle \{s\} = \B^n_{n-k}(A'_{i+1})$, 
where $A_{i+1}' := A_{i+1} \triangle \{s\}$. We claim that
\[ |\B^{(1)} \sm \A_{i}| + |\B^{(2)} \sm \A_{i}|
= |\B^{(1)} \sm \A_{i+1}| + |\B^{(2)} \sm \A_{i+1}|.\]
To see this, we consider the number of times 
that any set $A$ is counted by each side of the identity.
If $C_{\es,\{s\}}(A)=A$ then interchanging $\A_i$ and $\A_{i+1}$ 
does not affect the contribution of $A$. This remains true
when $C_{\es,\{s\}}(A) \ne A$, unless $A \in \A_i \sm \A_{i+1}$
and $A \in \A_{i+1} \sm \A_i$. In this last case, we note that 
$C_{\es,\{s\}}(\B^{(1)} \cup \B^{(2)}) = \B^{(1)} \cup \B^{(2)}$,
so $A$ contributes to the left hand side of the identity
iff $C_{\es,\{s\}}(A)$ contributes to the right hand side.
The claim follows.

As $|\B_{i+1} \sm \A_{i+1}| \leq \dD \tbinom {n-1}{k-1}$,
we deduce $|\B^{(1)} \sm \A_{i}| + |\B^{(2)} \sm \A_{i}| 
\leq \tbinom {n-1}{k-1} + 2 \dD \tbinom {n-1}{k-1}$, so
$|\B^{(j)} \sm \A_{i}| \leq \tfrac {1}{2} \tbinom {n-1}{k-1} 
+\dD \tbinom {n-1}{k-1}\leq \tbinom {n-1}{k-1} - \dD \tbinom {n-1}{k-1}$ 
for some $\B_i \in \{\B^{(1)},\B^{(2)}\}$ (as $\dD < \tfrac {1}{4}$).
Lemma \ref{vbcorrection2} improves this to 
$|\B_i \triangle \A_i| 	\leq \dD \tbinom {n-1}{k-1}$,
and so completes the proof.
\end{proof}

\subsection{Decompressing upsets}

Of the two extremal families in Theorem \ref{thm: VIP stability},
only one ($\G_1)$ is an upset. In this subsection we show that
any upset with small vertex boundary is approximated by such a family.	
	
\begin{lem} \label{lem: decompressing upsets}
Let $\dD \in (0, \tfrac {1}{3})$, $c = 10^{-9}\dD$, 
$k \geq 2$
and $\A \subset \{0,1\}^n$ be an upset with 
$|\A| = \tbinom {n}{\geq k+1} + \tbinom {x}{k}$ and
$|\pl_v(\A)| \leq \Lov {\A } 
+ \tfrac {ck(x-k)}{x^3} \tbinom {x}{k-1}$,
where $\tbinom {x}{k} = \tbinom {|S|}{k} 
\pm \tfrac {\dD}{5} \tbinom {|S|-3}{k-2}$ 
for some $|S| \in [k,n-1]$. 

Suppose that $U, V \subset [n]$ are disjoint sets 
with $|U| + 1 = |V| \geq 2$ and  $\B = C_{U,V}(\A)$ 
satisfies $|\B \triangle \G| \leq \dD \tbinom {|S|-3}{k-2}$, 
where $\G = \tbinom {[n]}{\geq k+1} \cup \tbinom {S}{k}$. 
Then $|\A \triangle \G| \leq \dD  \tbinom {|S|-3}{k-2}$.
\end{lem}

\begin{proof}
First we note that 
$|\A| = |\B| = |\G| \pm \dD \tbinom {|S|-3}{k-2}$,
so $|\G \sm \A| - |\G \sm \B|
\le |\A \sm \G| - |\B \sm \G| 
+ 2 \dD \tbinom {|S|-3}{k-2}$, and so 
$|\A \triangle \G| - |\B \triangle \G|
\le 2(|\A \sm \G| - |\B \sm \G|
+ \dD \tbinom {|S|-3}{k-2})$.
It will therefore suffice to bound $|\A \sm \G| - |\B \sm \G|$, 
which counts sets removed from $\G$ under the decompression,
i.e.\ $C_{U,V}(A) \in (\B \sm \A) \cap \G$ 
and $A \in (\A \sm \B) \sm \G$. Such sets must satisfy: 
\begin{enumerate}
\item[(a)] $C_{U,V}(A) \in (\B \sm \A) \cap \tbinom {[n]}{k+1}$ 
and $A \in \tbinom {[n]}{k} \sm \tbinom {S}{k}$, or 
\item[(b)] $C_{U,V}(A) \in (\B \sm \A) \cap \tbinom {S}{k}$ 
and $A \in (\A \sm \B) \cap \tbinom {[n]}{k-1}$.
\end{enumerate}
We write $\mc{T}_a$ or $\mc{T}_b$ for the families
of type (a) or (b) sets as above. When bounding $\mc{T}_a$,
it will be more convenient to bound
$\D := \tbinom{[n]}{\geq k+1} \sm \A$, noting that 
\[ \mc{T}_a \sub \D \sub \mc{T}_a \cup
(\tbinom{[n]}{\geq k+1} \sm \B).\] 

We divide the remainder of the proof into
cases according to the size of $S$.
We start with the case $|S| \leq n-3$.
As $|U| + 1 = |V|$ we have
$\big | |A| - |C_{U,V}(A)| \big | \leq 1$
for any set $A$, so 
\[ |\A \sm \tbinom{[n]}{\geq k+1}| \leq 
|\B\sm \tbinom{[n]}{\geq k+1}| + \tbinom {n-|U|-|V|}{k+1-|V|} 
\leq \tbinom {|S|}{k} + \dD \tbinom {|S|-3}{k-2} + \tbinom {n-3}{k-1}  
\leq \tbinom {n-1}{k} - \dD \tbinom {|S|-3}{k-2},\] 
as $\dD < \tfrac {1}{2}$. By Lemma \ref{vbcorrection1} we deduce
$|\mc{T}_a| \le |\D| \leq \dD \tbinom {|S|-3}{k-2}$.

To bound type (b) sets, we define an injection from
$\mc{T}_b$ to $\A \cap \big ( \tbinom {[n]}{k} \sm \tbinom {S}{k} \big )$
by $A \mapsto A+s$, for some fixed $s \in [n]$ with
$s \in S^c$ if $U \subset S$ or $s \in V$ if $U \not \subset S$. 
To see that this map is well-defined on $A \in \mc{T}_b$, 
note that $A+s \in \A$ as $\A$ is an upset, and $s \notin A$ using 
$A \subset C_{U,V}(A) \cup U \subset S$ if $U \subset S$ 
or $A \cap V = \es$ if $U \not \subset S$. We also note that
\[ |\A \cap \big ( \tbinom {[n]}{k} \sm \tbinom {S}{k} \big )|
\le |\D| + |\B \sm \G|, \]
as if $A \in \A \cap \big ( \tbinom {[n]}{k} \sm \tbinom {S}{k} \big )$
we have $A \in \B \sm \G$ or $C_{U,V}(A) \in \D$.
We deduce $|\mc{T}_b| \le |\D| + |\B \sm \G| 
\le 2 \dD \tbinom {|S|-3}{k-2}$, so
$|\A \sm \G| - |\B \sm \G| \leq 3\dD  \tbinom {|S|-1}{k-1}$,
giving $|\A \triangle \G| \le 8\dD \tbinom {|S|-3}{k-2}$.
Lemma \ref{localstabcor} improves this to the required bound
$|\A \triangle \G| \leq \dD \tbinom {|S|-3}{k-3}$,
which completes the proof if $|S| \le n-3$.		
			
Henceforth we can assume $|S| \in \{n-2,n-1\}$.
Next we consider the case $U \cap S^c \neq \es$.
As $U \cap V = \es$ we have $|V \cap S^c| \le 1$.
We start by bounding type (a) sets according
to the two subcases $|V \cap S^c| = 0,1$.
First we consider the subcase $V \cap S^c = \{v\}$,
in which case we can define an injection from $\mc{T}_a$
to $\tbinom{S}{k} \sm \B$ by $A \mapsto C_{U,V-v}(A)$.
Indeed, as $U \sub A$ and $A \cap V = \es$ we have 
$C_{U,V-v}(A) \in \tbinom{S}{k}$. Furthermore,
$C_{U,V-v}(A) \notin \B$, as otherwise 
$C_{U,V-v}(A) \in \A$ but $C_{U,V}(A) \notin \A$,
which contradicts $\A$ being an upset.
We deduce $|\mc{T}_a| \le |\tbinom{S}{k} \sm \B|
\le \dD \tbinom {|S|-3}{k-3}$ 
in the subcase $|V \cap S^c|=1$.

Now consider the subcase $V \cap S^c = \es$.
The same argument as in the previous subcase
(using any $v \in V$) bounds the number of
$A \in \mc{T}_a$ with $C_{U,V}(A) \sub S$.
This accounts for all type (a) sets if $|S|=n-1$.
If $|S|=n-2$ then any further sets $A \in \mc{T}_a$ 
contain $S^c$, so number at most
$\tbinom{n-|V|-|U|-1}{k-|U|-1}
\le \tbinom{n-4}{k-2}$. We deduce
$|\mc{T}_a| \le \tbinom{n-4}{k-2} 
+ \dD \tbinom {|S|-3}{k-3}$,
so $|\D| \le |\mc{T}_a| + \dD \tbinom {|S|-3}{k-3}
\le \tbinom{n-4}{k-2} + 2\dD \tbinom {|S|-3}{k-3}
\le \tbinom {n-1}{k-1} - \dD \tbinom {|S|-3}{k-2}$
as $\dD < \tfrac {1}{3}$. By Lemma \ref{vbcorrection1} we deduce
$|\mc{T}_a| \le |\D| \leq \dD \tbinom {|S|-3}{k-2}$,
thus bounding type (a) sets in both subcases.

Now we can bound type (b) sets by the same argument
as in the case $|S| \le n-3$, using an injection
$\mc{T}_b \to \A \cap ( \tbinom {[n]}{k} \sm \tbinom {S}{k})$
defined by $A \mapsto A+v$ for any fixed $v \in V$.
To see that this is well-defined on $A \in \mc{T}_b$,
note that $v \notin A$ as $C_{U,V}(A) \ne A$,
and that $U \sub A \not \sub S$. The remainder 
of the proof follows as in the previous case,
so henceforth we can assume $|S| \in \{n-2,n-1\}$
and $U \cap S^c = \es$.

We can assume $S^c \not \sub V$, as otherwise 
$|A \triangle \G| = |\B \triangle \G| 
\leq \dD \tbinom {|S|-1}{k-1}$.	
To see this, note that $C_{U,V}(A) = A$ for any
$A \in \tbinom {[n]}{k} \sm \tbinom {S}{k}$ 
as $A \cap V \neq \es $, and that no 
$A \in \A \cap \tbinom {[n]}{k-1}$ has 
$C_{U,V}(A) \in \tbinom {S}{k}$, as $V \cap S^c \neq \es$.

Without loss of generality, $n \in S^c \sm V$.
As in \eqref{eqsec} we use superscripts $0$ and $1$
to denote the sections of a family in direction $n$. 
Note that $A$ and $C_{U,V}(A)$ belong to the same section
for any set $A$, as $n \notin U \cup V$.
This gives $|\A^1| = |\B^1| = \tbinom {n-1}{\geq k} 
\pm \dD \tbinom {|S|-3}{k-2}$ and
$|\A^0| = |\B^0| = \tbinom {n-1}{\geq k+1} 
+ \tbinom {|S|}{k} \pm \dD \tbinom {|S|-1}{k-1}
\geq \tbinom {n-1}{\geq k+1} 
+ \tbinom {x}{k} - 2\dD \tbinom {x-1}{k-1}
\ge \tbinom {n-1}{\geq k+1} + \tbinom {x-1}{k}$. 
Note that if $k = 2$ we have
$|{\cal A}^1| = \binom {n-1}{\geq 2}$.
Furthermore, as $\A$ is an upset we have
$\A^0 \sub \A^1$, so $|\A^0| \le |\A^1|$.
Lemma \ref{sections} therefore gives $|\pl_v^{(n-1)}(\A^1)| \leq
 \LovDim {n-1}{\A ^1} + \tfrac {ck(x-k)}{x^3} \tbinom {x}{k-1}
\le  \LovDim {n-1}{\A ^1} + \tfrac {c}{n-1} \tbinom {n-1}{k-1}$. 
Then Theorem \ref{thm: ball-sized stability} gives 
$|\A^1 \triangle {\cal H}| \leq \dD \tbinom {n-2}{k-1}$ 
for some Hamming ball ${\cal H} \subset \{0,1\}^{n-1}$,
and Lemma \ref{vbcorrection2} improves this to
$|\A^1 \triangle {\cal H}| \leq \dD \tbinom {|S|-3}{k-2}$.
As $\A^1$ is an upset, ${\cal H} = \tbinom {n-1}{\geq k}$. 
 
In particular, the number of type (a) and type (b) sets 
containing $n$ are both bounded by $\dD \tbinom {|S|-3}{k-2}$.
As $\A^0 \subset \A^1$ we have  
$|\A^0 \sm {\cal H}| \leq \dD \tbinom {|S|-1}{k-1}$. 
In particular, this bounds type (b) sets in $\A^0$. If $|S|=n-1$ 
then $|\A \sm \G| = |\A^0\sm {\cal H}| + |\A^1\sm {\cal H}| 
\leq 2 \dD  \tbinom {|S|-1}{k-1}$, and Lemma \ref{localstabcor}
improves this to $|\A \triangle \G| \leq \dD \tbinom {|S|-3}{k-2}$.

Finally, we consider $|S|=n-2$ and bound type (a) sets in $\A^0$.
We write $[n-1] \sm S = \{v\}$ and define an injection 
$A \mapsto C_{U,V}(A) - v$ from $\mc{T}_a \cap \A^0$ to 
$(\tbinom{S}{k} \sm \B) \cup C_{U,V}(\A^0 \sm \mc{H})$.
To see that this is well-defined, first note that
$A \in \tbinom{[n-1]}{k}$ and $C_{U,V}(A) \ne A$,
so $v \in A \sm U$ and $v \in C_{U,V}(A) \in \B \sm \A$.
As $\A$ is an upset, $C_{U,V}(A)-v \in \tbinom{S}{k} \sm \A$.
If $C_{U,V}(A)-v \notin \tbinom{S}{k} \sm \B$ then
$C_{U,V}(A-v) = C_{U,V}(A)-v \in \B \sm \A$,
so $A-v \in \A^0 \sm \mc{H}$. We deduce $|\mc{T}_a \cap \A^0| 
\le |\tbinom{S}{k} \sm \B| + |\A^0 \sm \mc{H}| 
\le 2 \dD \tbinom {|S|-3}{k-2}$.
Altogether, $|\A \sm \G| - |\B \sm \G|
\le 5 \dD \tbinom {|S|-3}{k-2}$,
so $|\A \triangle \G| \le 12 \dD \tbinom {|S|-3}{k-2}$,
and Lemma \ref{localstabcor} improves this to 
$|\A \triangle \G| \leq \dD \tbinom {|S|-3}{k-2}$.
\end{proof}

\subsection{Decompressing general sets}					
			
In this subsection we prove that if $\A$ 
has small vertex boundary and $C_{\es ,\{i\}}(\A)$ 
is close to a generalised Hamming ball then so is $\A$.
Without loss of generality we take $i=n$.
First we show that the size of the intersection
of two Hamming balls is a non-increasing function
of the distance between their centres.
At first, this may sound too obvious to need a proof,
but perhaps surprisingly, if $t$ is odd then increasing
the distance from $t$ to $t+1$ makes no difference
to the intersection size.

\begin{lem} \label{monotone}
Let $f_t(n,k) = |\B^n_{n-k}(C) \cap \B^n_{n-k}(C')|$ 
where $|C \triangle C'|=t$. 
Let $\mc{D}_t(n,k) = 
\{A \sub [n-1]: |A|=|A \triangle [t]|=k-1\}$.
Then $f_t(n,k) - f_{t+1}(n,k) = |\mc{D}_t(n,k)|$.
\end{lem}

\begin{proof}
We write $f_t(n,k) - f_{t+1}(n,k)
= |\B^n_{n-k}([n]) \cap \B^n_{n-k}([n] \sm [t])|
- |\B^n_{n-k}([n]) \cap \B^n_{n-k}([n] \sm [t+1])|
= |\mc{X}'|-|\mc{X}|$, where 
$\mc{X}' = |\{A' \sub [n]: |A'| \ge k,
 |A' \triangle [t]| = k,
 |A' \triangle [t+1]| = k-1\}|$ and $\mc{X} = |\{A \sub [n]: |A| \ge k,
 |A \triangle [t+1]| = k,
 |A \triangle [t]| = k-1\}|$.
Every set $A \in \mc{X}$ does not contain $t+1$,
and adding $t+1$ gives a set $A' \in \mc{X}'$.
The map $A \mapsto A \cup \{t+1\}$ is injective, 
so $|\mc{X}'|-|\mc{X}|$ is the number of sets in $\mc{X}'$ 
not in the image, i.e.\ $|\mc{X}'|-|\mc{X}| 
= |\{A: t+1 \in A, |A|=|A \triangle [t]|=k\}| = |\mc{D}_t(n,k)|$.
\end{proof}
 
Now we come to the main lemma of this subsection.

\begin{lem} \label{lem: gen ball approx}
Let $\dD \in (0,1)$, $c = 10^{-9}\dD$
and $\A \subset \{0,1\}^n$ with 
$|\A| = \tbinom {n}{\geq k+1} + \tbinom {x}{k}$ 
and $|\pl_v(\A)| \leq \Lov {\A } 
+ \tfrac {c k(x-k)}{x^3} \tbinom {x}{k-1}$, 
where $\tbinom {x}{k} = \tbinom {|S|}{k} 
\pm \tfrac {\dD }{8} \tbinom {|S|-3}{k-2}$ 
for some $|S| \in [k,n-1]$ with $k\geq 2$. 
Suppose $\B := C_{\es,\{n\}}(\A)$ 
satisfies $|\B \triangle \G| \leq \dD \tbinom {|S|-3}{k-2}$ 
for some generalised Hamming ball $\G$ with 
$|\G| = \tbinom {n}{\geq k+1} + \tbinom {|S|}{k}$. 
Then $|\A \triangle \G'| \leq \dD \tbinom{|S|-3}{k-2}$ 
for some generalised Hamming ball $\G'$.
\end{lem}

\begin{proof}
First we note that the lemma is trivial for $k \ge n-1$,
so we can assume $k \le n-2$.
By applying an automorphism of the cube,
we may assume $\G = \G_1 = \tbinom {[n]}{\geq k+1} \cup \tbinom {S}{k}$ 
or $\G = \G_2 = \tbinom {n}{\geq k+1} \cup 
\tbinom {S'}{k} \cup \tbinom {S'}{k-1}$ with $|S| = |S'| + 1$. 
These two cases are in turn each split into two subcases
according to whether $n$ belongs to $S$ or $S'$, denoted
by superscripts as in \eqref{eqsec}, as follows:
\begin{itemize}
\item[({a})]  
$\G_{1}^0 = \tbinom {[n-1]}{\geq k+1} \cup \tbinom {S}{k}$ and 
$\G^1_{1} = \tbinom {[n-1]}{\geq k+1} \cup \tbinom {[n-1]}{k}$,
where $n \notin S$; 
\item[({b})]
$\G_{1}^0 = \tbinom {[n-1]}{\geq k+1} \cup \tbinom {S'}{k}$ and 
$\G_{1}^1 = \tbinom {[n-1]}{\geq k+1} \cup 
 \tbinom {[n-1]}{k}\cup \tbinom {S'}{k-1}$, 
where $S = S' \cup \{n\}$;
\item[({c})]  
$\G_{2}^0 = \tbinom {[n-1]}{\geq k+1} 
 \cup \tbinom {S'}{k} \cup \tbinom {S'}{k-1}$ and 
$\G^1_{2} = \tbinom {[n-1]}{\geq k+1}  \cup \tbinom {[n-1]}{k}$,
where 	$n \notin S'$; 
\item[({d})]
$\G_{2}^0 = \tbinom {[n-1]}{\geq k+1} \cup 
 \tbinom {S'' }{k} \cup \tbinom {S'' }{k-1}$ and 
$\G_{2}^1 = \tbinom {[n-1]}{\geq k+1} \cup \tbinom {[n-1]}{k} 
 \cup \tbinom {S''}{k-1} \cup \tbinom {S'' }{k-2}$,
 where $S' = S'' \cup \{n\}$.
\end{itemize}
A family is of type (a) if 
it can be approximated up to $\dD \tbinom {|S|-3}{k-2}$ 
elements by a family {isomorphic} to (a), 
and similarly for type (b), (c), (d). 
Some case-checking shows that then the type and the associated set 
$S$, $S'$ or $S''$ are unique (which we omit,
as we do not use this fact in the proof).
We let $\G^0$ and $\G^1$ denote the appropriate
families for the approximation of $\B$.

As $\B = C_{\es,\{n\}}(\A)$, we note that $\A$ and $\B$
are related by the `intersection-union transformation'
\[ \B^0 = \A^0 \cap A^1 \quad \text{ and } \quad \B^1 = \A^0 \cup A^1. \]
In particular, $\B^0 \sub \B^1$, so $\B$ cannot be of type (c),
which has $|\G^0 \sm \G^1| = \tbinom{|S|-1}{k-1} 
> \dD \tbinom{|S|-3}{k-2} \ge |\B \triangle \G|$.
By possibly swapping $\A^0$ and $\A^1$
we can assume $|\A^0| \le |\A^1|$; indeed,
this does not affect $\B^0$ and $\B^1$,
and any approximation for the `swapped' family
gives one for $\A$, via the automorphism of the cube
that swaps $0$ and $1$ in coordinate $n$.
We claim that the sections of $\A$ have just two possible
types of approximate sizes, namely
\begin{enumerate}[(i)]
\item $|\A^0| = \tbinom {n-1}{\geq k+1} + \tbinom {|S|}{k} 
\pm \dD \tbinom {|S|-3}{k-2}$ 
and $|\A^1| = \tbinom {n-1}{\geq k+1} + \tbinom {n-1}{k} \pm 
\dD \tbinom {|S|-3}{k-2}$, or
\item $|\A^0| = \tbinom {n-1}{\geq k+1} + \tbinom {|S|-1}{k} 
\pm \dD \tbinom {|S|-3}{k-2}$ and 
$|\A^1| = \tbinom {n-1}{\geq k+1} + \tbinom {n-1}{k} 
+ \tbinom {|S|-1}{k-1} \pm \dD \tbinom {|S|-3}{k-2}$.			
\end{enumerate}
To see this claim, first note that 
\[ |\A^0| \ge |\B^0| 
\ge |\G^0|-\dD \tbinom {|S|-3}{k-2}
\ge \tbinom {n-1}{\geq k+1} + \tbinom {|S|-1}{k} 
- \dD \tbinom {|S|-3}{k-2}.\]
If (ii) does not hold then 
$|\A^0| > \tbinom {n-1}{\geq k+1} + \tbinom {|S|-1}{k} 
+ \dD \tbinom {|S|-3}{k-2} > \tbinom {n-1}{\geq k+1} 
+ \tbinom {x-1}{k} + \tfrac{\dD}{2} \tbinom {x-3}{k-2}$
(the latter by Lemma \ref{shiftapprox}), so (i) holds
by Lemma \ref{sections} 
(applied with $\tfrac{\dD}{2}$ in place of $\dD$).

We now consider separate cases according to whether the type 
of the sizes of the sections of $\A$ is the same as that of $\B$.
Suppose first that $|\A^0| = |\G^0| \pm \dD \tbinom {|S|-3}{k-2}$
(which is the same estimate that we know for $|\B^0|$).
Then $|\A^1| = |\G^1| \pm 2\dD \tbinom {|S|-3}{k-2}$.
As $\B^0 \sub \A^0$ we have 
$|G^0 \sm \A^0| \le |\G^0 \sm \B^0| \le \dD \tbinom {|S|-3}{k-2}$,
so $|\G^0 \triangle \A^0| \le 2|\G^0 \sm \A^0| 
+ ||\A^0|-|\G^0|| \le 3 \dD \tbinom {|S|-3}{k-2}$.
Similarly, $|A^1 \sm \G^1| \le |\B^1 \sm \G^1| 
\le \dD \tbinom {|S|-3}{k-2}$, 
so $|\G^1 \triangle \A^1| \le
2|A^1 \sm \G^1| + ||\A^1|-|\G^1|| 
\le 4 \dD \tbinom {|S|-3}{k-2}$.
We deduce $|\A \triangle \G| \le 7\dD \tbinom {|S|-3}{k-2}$.
Lemma \ref{localstabcor} improves this to
$|\A \triangle \G| \le \dD \tbinom {|S|-3}{k-2}$,
so $\A$ has the same type as $\B$,
and the proof is complete in this case.

It remains to consider the case 
$|\A^0| \notin |\G^0| \pm \dD \tbinom {|S|-3}{k-2}$,
i.e.\ the sizes of the sections of $\A$ 
are of the opposite type to those of $\B$.
Here we note that $\B$ must be of type (b) or (d).
Indeed, we have already noted that (c) is impossible,
and type (a) falls into the previous case,
as $|\A^0| \ge |\G^0| - \dD \tbinom {|S|-3}{k-2}
= \tbinom {n-1}{\geq k+1} + \tbinom {|S|}{k} 
- \dD \tbinom {|S|-3}{k-2}
> \tbinom {n-1}{\geq k+1} + \tbinom {|S|-1}{k} 
+ \dD \tbinom {|S|-3}{k-2}$.
Thus $\B$ has section sizes of type (ii)
and $\A$ has sections sizes of type (i).

By Lemma \ref{sections}, 
$|\pl^{(n-1)}_v(\A^j)| \le \LovDim{n-1}{\A^j}
+ \tfrac {c k(x-k)}{x^3} \tbinom {x}{k-1}$ for $j=0,1$.
As $|\A^1| = \tbinom {n-1}{\geq k} \pm \dD \tbinom{|S|-3}{k-2}$ 
(giving $|\A^1| = \tbinom {n-1}{\geq k}$ if $k=2$),
by Theorem \ref{thm: ball-sized stability} 
we have $|\A^1 \triangle \mc{H}^1| < 5\dD \tbinom{|S|-3}{k-2}$
for some Hamming ball $\mc{H}^1$ in $\{0,1\}^{n-1}$
of size $\tbinom {n-1}{\geq k}$. Now we see that $\B$ 
cannot be of type (d), as this would give
$|\G \sm \B| \ge |\G^0_2 \sm \A^1| \ge \tbinom{|S|-2}{k-1}
- 5\dD \tbinom{|S|-3}{k-2} > \dD \tbinom{|S|-3}{k-2}$, contradiction.
Thus $\B$ has type (b), i.e.\ 
$|\B \triangle \G_1| \le \dD \tbinom{|S|-3}{k-2}$
with $\G_{1}^0 = \tbinom {[n-1]}{\geq k+1} \cup \tbinom {S'}{k}$ 
and $\G_{1}^1 = \tbinom {[n-1]}{\geq k+1} \cup 
 \tbinom {[n-1]}{k}\cup \tbinom {S'}{k-1}$, 
where $S = S' \cup \{n\}$.

Next we consider the subcase that
$\tbinom{|S|-1}{k-1} \le \tbinom{n-2}{k-1} - 7\dD \tbinom{n-4}{k-2}$.
We must have $\mc{H}^1 = \tbinom {[n-1]}{\geq k}$, as otherwise
by Lemma \ref{monotone} we get $|\B^1 \sm \G^1| \ge |\A^1 \sm \G^1|
\ge \tbinom{n-2}{k-1} - \tbinom{|S|-1}{k-1} - 5\dD \tbinom{|S|-3}{k-2}
\ge 2 \dD \tbinom{|S|-3}{k-2}$, contradiction.
As $\A^0 \sm \mc{H}^1 \sub \B_1 \sm \mc{H}^1$
we have $|\A^0 \sm \G^1_1| \le \dD \tbinom{|S|-3}{k-2}$,
and as $\mc{H}^1 \sm \A^0 \sub \mc{H}^1 \sm \B^0$
we have $|\G^0_1 \sm \A^0| \le \dD \tbinom{|S|-3}{k-2}$.
Then with $\G_{2}^0 = \tbinom {[n-1]}{\geq k+1} 
 \cup \tbinom {S'}{k} \cup \tbinom {S'}{k-1}$ as in (c)
we have $|\A^0 \triangle \G_2^0| \le 5 \dD \tbinom{|S|-3}{k-2}$
(using $|\A^0|=|\G_2^0| \pm \dD \tbinom {|S|-3}{k-2}$)
so $|\A \triangle \G_2| \le 10 \dD \tbinom{|S|-3}{k-2}$.
Lemma \ref{localstabcor} improves this to
$|\A \triangle \G_2| \le \dD \tbinom {|S|-3}{k-2}$,
so $\A$ has type (c), which completes the proof of this subcase.

It remains to consider the subcase that
$\tbinom{|S|-1}{k-1} > \tbinom{n-2}{k-1} - 7\dD \tbinom{n-4}{k-2}
= (1 - \tfrac{7\dD(k-1)(n-k-1)}{(n-2)(n-3)}) \tbinom{n-2}{k-1}$.
By Lemma \ref{shiftapprox} we have $\tbinom{|S|}{k-1} > 
(1 - \tfrac{7\dD(k-1)(n-k-1)}{(n-2)(n-3)}) \tbinom{n-1}{k-1}
> \tbinom{n-1}{k-1} - 7\dD \tbinom {n-3}{k-2}$.
Then $|\A^1| =\tbinom {n-1}{\geq k} \pm \dD \tbinom {n-3}{k-2}$
and $|\A^0| = \tbinom {n-1}{\geq k} \pm 7\dD \tbinom {n-3}{k-2}$.
By Lemma \ref{sections}, 
$|\pl^{(n-1)}_v(\A^j)| \le \LovDim{n-1}{\A^j}
+ \tfrac {c k(x-k)}{x^3} \tbinom {x}{k-1}$ for $j=0,1$,
so by Theorem \ref{thm: ball-sized stability} 
we have $|\A^1 \triangle \mc{H}^1| < 5\dD \tbinom{n-3}{k-2}$
and $|\A^0 \triangle \mc{H}^0| < 35\dD \tbinom{n-3}{k-2}$
for some Hamming balls $\mc{H}^0,\mc{H}^1$ in $\{0,1\}^{n-1}$
both of size $\tbinom {n-1}{\geq k}$.
Note that $\mc{H}^0 \ne \mc{H}^1$, as otherwise
we would be in our previous case where $\A$ and $\B$
have the same type of section sizes.

Next we claim that
the centres of $\mc{H}^0$ and $\mc{H}^1$
cannot be at distance more than $1$ apart.
To see this, first note that either centre is at distance
at most $2$ from $[n]$, as otherwise by Lemma \ref{monotone} 
we get $|\B^1 \sm \tbinom {[n-1]}{\geq k}|
\ge \tbinom{n-2}{k-1} + 2 \tbinom{n-3}{k-2}$,
so $|\B^1 \sm \G^1|\ge (2-\dD) \tbinom{n-3}{k-2}$, contradiction.
Furthermore, we cannot have either centre at distance
exactly $2$ from $[n]$, say $\mc{H}^i=B^{n-1}_{n-k-1}([n-2])$,
as then $\tbinom {[n-1]}{\geq k} \sm \mc{H}^i$
contains $\{A \sub [n-1]: |A|=k+1, \{n-1,n-2\} \sub A\}$
of size $\tbinom{n-3}{k-1} \ge \tbinom{|S|-3}{k-2}$,
so $|\G^0 \sm \B^0|\ge (1-\dD) \tbinom{|S|-3}{k-2}$, contradiction.
It remains to rule out two centres of size $n-1$,
say $\mc{H}^i=B^{n-1}_{n-k-1}([n-1] \sm \{x_i\})$ for $i=0,1$.
In this case $\mc{H}^0 \cup \mc{H}^1$ has no sets of size $k-2$,
which rules out $\B$ of type (d), which has $\tbinom{|S|-2}{k-2}
> \tbinom{|S|-3}{k-2}$ such sets. Also, $\mc{H}^0 \cap \mc{H}^1$ 
contains all sets of size $k-1$ disjoint from $\{x_0,x_1\}$;
there are $\tbinom{n-3}{k-1} \ge \tbinom{|S|-3}{k-2}$ such sets,
which rules out $\B$ of type (b), and so proves the claim.

We conclude that the centres of $\mc{H}^0$ and $\mc{H}^1$ are at distance 1.
Let $\mc{H} \sub \{0,1\}^n$ have sections $\mc{H}^0,\mc{H}^1$.
Then $\mc{H}$ is isomorphic to a generalised Hamming ball
$\G' = \tbinom {[n]}{\geq k+1} \cup \tbinom {[n-2]}{k}
\cup \tbinom {[n-2]}{k-1}$. We have
$|\A \triangle \G'| < 40\dD \tbinom{n-4}{k-2}$, and
Lemma \ref{localstabcor} improves this to the required approximation
$|\A \triangle \G'| < \dD \tbinom{n-4}{k-2}$. 
\end{proof}
						
\subsection{Stability for Harper's Theorem}
	
We conclude this section by proving our main result
on stability for vertex isoperimetry in the cube.
	
\begin{proof}[Proof of Theorem \ref{thm: VIP stability}]
Let $\dD \in (0,1)$, $c = 10^{-10} \dD$ and $\A \sub \{0,1\}^n$ 
with $|\A| = \tbinom {n}{\geq k+1} + \tbinom {x}{k}$ and 
$|\pl_v(\A)| \leq \Lov {\A } + \tfrac {ck(x-k)}{x^3} \tbinom {x}{k-1}$.
Let $\{(U_i,V_i)\}_{i\in [L_1]}$ and $\{\A_i\}_{i\in [L_1]}$ 
be as in Theorem \ref{thm: compression theorem}. 
We will show for $L_1 \ge i \ge 0$ that there is some 
generalised Hamming ball $\G_i$ with
$|\G_{i} \triangle \A_i| \leq \dD \tbinom {|S|-1}{k-1}$.
As $\A_0 = \A$, the theorem will follow by taking $i = 0$.

We start by considering $\A_{L_1}$, which is `ball-like',
i.e.\ $\A_{L_1} = \tbinom {[n]}{\geq k+1} \cup \B$,
for some $\B \subset \tbinom {[n]}{k}$. 
As $|\A_{L_1}| = |\A|$, we have $|\B| = \tbinom {x}{k}$.
Theorem \ref{thm: compression theorem}.iv gives 
$\tbinom {n}{k} - \tbinom {x}{k} + |\pl(\B)| 
= |\pl_v(\A_{L_1})| \leq |\pl_v(\A_0)| 
\leq  \tbinom {n}{k} - \tbinom {x}{k} 
+ \big ( 1 + \tfrac {ck(x-k)}{x^3} \big ) \tbinom {x}{k-1}$,
so $|\pl(\B)| \leq \big ( 1 + \tfrac {ck(x-k)}{x^3} \big ) \tbinom {x}{k-1}$.  
By Theorem \ref{thm: KK-stab}
(with $\tfrac {ck(x-k)}{x^2}$ in place of $c$)
we have $|\B \triangle \tbinom {S}{k}| 
\leq \tfrac {\dD }{8} \tbinom {|S|-3}{k-2}$ 
for some $S \subset [n]$, so $|\A_{L_1} \triangle \G| 
\leq \tfrac {\dD}{8} \tbinom {|S|-3} {k-2}$,
where $\G = \tbinom {[n]}{\geq k+1} \cup \tbinom {S}{k}$.
Note that $\tbinom {x}{k} = |\B| =
\tbinom {|S|}{k} \pm \tfrac {\dD }{8} \tbinom {|S|-3}{k-2}$. 
If $|S|=n$ then the theorem follows from
Theorem \ref{thm: ball-sized stability} applied to $\A$
(with $\tfrac{2ck(n-k)}{n^2}$ in place of $c$) 
so we may assume $|S| \leq n-1$.
			
Next we show $|\A_i \triangle \G| \leq \dD  \tbinom {|S|-3}{k-2}$ 
for $L_1 \ge i \ge L_0$. The case $i=L_1$ was proved above.
We proceed inductively for $i<L_1$, supposing the required
approximation for $\A_{i+1}$. As $\A_i$ is an upset 
with $|\A_i| = \tbinom {n}{\geq k+1} + \tbinom {x}{k}$ 
and $|\pl_v(\A_i)| \leq \Lov {\A _i} 
+ \tfrac {ck(x-k)}{x^3} \tbinom {x}{k-1}$,
by Lemma \ref{lem: decompressing upsets} we have
$|\A_i \triangle \G| \leq \dD \tbinom {|S|-3}{k-2}$, as required.

To complete the proof, we now show for $L_0 \ge i \ge 0$ that  
there is a generalised Hamming ball $\G_i$ with 
$|\G_i \triangle \A_i| \leq \dD \tbinom {|S|-3}{k-2}$. 
We showed this above for $i = L_0$. 
Proceeding inductively for $i<L_0$,
given the required approximation 
$|\G_{i+1} \triangle \A_{i+1}| \leq \dD \tbinom {|S|-3}{k-2}$
for $\A_{i+1}$, by Lemma \ref{lem: gen ball approx}
we have $|\A_i \triangle \G_i| \leq \dD \tbinom{|S|-3}{k-2}$ 
for some generalised Hamming ball $\G_i$, as required.
\end{proof}

\section{Applications}

In this section we give various applications of our stability
versions of Harper's Theorem and Kruskal--Katona to stability
versions of other results in Extremal Combinatorics.
We start with stability for the Erd\H{o}s--Ko--Rado theorem.
First we recall an 
estimate on shadows known as the `LYM inequality' (see \cite{Bol}):
if $n\geq k \geq 1$ and $\A \sub \tbinom{[n]}{k}$ with $|\A| = \aA \tbinom{n}{k}$
then $|\pl(\A)| \ge \aA \tbinom{n}{k-1}$. This estimate is weaker
than those used elsewhere in the paper but will be convenient 
in some calculations. We will use it in the following form that
follows from Kruskal--Katona, Lemma \ref{kkprops}.i and LYM:
\begin{equation}\label{LYM+}
|\A| = \tbinom{n-1}{k} + \aA \tbinom{n-1}{k-1} \Ra
|\pl(\A)| \ge \tbinom{n-1}{k-1} + \aA \tbinom{n-1}{k-2}.
\end{equation}

\begin{proof}[Proof of Theorem \ref{EKR-stability}] 
We apply a stability analysis to Daykin's proof \cite{DaykinEKR} 
of the Erd\H{o}s--Ko--Rado theorem.
Suppose $\A \subset \tbinom {[n]}{k}$ is intersecting.
Let $\B_{n-k} = \{A^c: A \in \A\}$ and iteratively define 
$\B_{i} := \pl (\B_{i+1})  \subset \tbinom {[n]}{i}$ for $n-k-1 \ge i \ge k$. 
Note that $\A \cap \B_{k} = \es$, as if $A \in \A \cap \B_{k}$ 
then there is $B\in \B_{n-k}$ with $A \subset B$,
i.e.\ $B^c \in \A$ with $A \cap B^c = \es$,
which contradicts $\A$ being intersecting.
In particular, $|\A| + |\B_k| \leq \tbinom {n}{k}$.
To prove the theorem, we will show that 
if $|\A|$ is close to $\tbinom {n-1}{k-1}$ then
this inequality is only possible when $\A$ is close to a star. 

Let $c_0 = 10^{-9} \tT$, $c = 10^{-3}c_0$ and $\dD = \tfrac {c(n-2k)}{n}$. 
Suppose $|\A| > (1-\dD) \tbinom{[n]}{k}$. We may assume $n \geq 16$,
as otherwise $|\A|=\tbinom{n-1}{k-1}$, so $\A$ is a star
by the characterisation of equality in the Erd\H{o}s--Ko--Rado theorem.
Define $x_i \ge k$ by $|\B_i| = \tbinom{x_i}{i}$ for all $i\in [k,n-k]$. 
Note that $x_{i} \geq x_{i+1}$ for $k \leq i <n-k$
by the Lov\'asz form of Kruskal--Katona.
Also, $\tbinom {x_{n-k}}{n-k} = |\B_{n-k}| =
|\A| \geq (1-\dD )\tbinom {n-1}{n-k} > (1+2\dD )^{-1} \tbinom {n-1}{n-k}$.
As $n-k \ge n/2$ this implies
$(1 + \tfrac{4\dD }{n})^{n-k} \tbinom {x_{n-k}}{n-k} > \tbinom {n-1}{n-k}$,
and so by Lemma \ref{binomratio}.i we deduce
$n-1 \le  (1 + \tfrac{4\dD }{n})x_{n-k} \le x_{n-k} + 4\dD$.

We claim that $|\pl(\B_{\ell })| \leq (1 + \tfrac{c_0}{n}) \tbinom {x_{\ell }}{\ell -1}$ 
for some $\ell \in [k,\min(n-k-1,3n/4)]$. Suppose for a contradiction that this fails.
As $x_{\ell }  \geq n-2 \geq 7n/8 \geq (1 + 1/6)\ell $ for all such $\ell$, 
by Lemma \ref{binomratio}.ii applied with $\alpha = 1/6$ and $\tT = \tfrac{c_0}{n}$ 
we have $x_{\ell} \geq (1 + \tfrac {c_0}{15n^2})x_{\ell +1}$. 
Applying this bound iteratively, as $\min (n-2k, n/4 + (n/2 -k)) \geq (n-2k)/2$ 
we obtain $x_k \geq (1 + \tfrac {c_0(n-2k)}{30 n^2}) x_{n-k}$. 
As $x_{n-k} \geq n-1 - 4\dD \geq \tfrac {7n}{8}$ this gives 
$x_k \geq n-1-4\dD + \tfrac {c_0}{40} \cdot \tfrac {n-2k}{n} \ge n-1 + 4\dD$. 
By Lemma \ref{binomratio}.i we deduce $|\B_k| = \tbinom {x_{k}}{k} \geq 
\big ( 1 + \tfrac {4\dD  k}{n}\big ) \tbinom {n-1}{k} 
=  \tbinom {n-1}{k} + \tfrac {4\dD (n-k)}{n} \tbinom {n-1}{k-1}
\geq \tbinom {n-1}{k} + {2\dD  } \tbinom {n-1}{k-1}$ as $k < n/2$. 
This contradicts $\B_k \cap \A = \es$, so the claim holds.

By Theorem \ref{thm: KK-stab}, there is $S \subset [n]$ 
with $|S| \in \{\lfloor x_{\ell }\rfloor, \lceil x_{\ell } \rceil \} \subset \{n-2,n-1,n\}$ 
so that $|\B_{\ell } \triangle \tbinom {S}{\ell }| \leq \tT \tbinom {|S|-1}{\ell -1}$. 
We claim that $|S| = n-1$. To see this, first note that 
$\tbinom {x_{\ell }}{\ell } \leq \tbinom {|S|}{\ell } + \tT \tbinom {|S|-1}{\ell -1} 
\leq \tbinom {|S| + \tT}{\ell }$ by \eqref{mvt}, so $|S| \geq x_{\ell } - \tT > n-2$.
On the other hand, if 
$|\tbinom {[n]}{\ell } \sm \B_{\ell }| \leq \tT \tbinom {n-1}{\ell -1}$ 
then  $|\B_{\ell }| \geq \tbinom {n-1}{\ell } + (1 - \tT) \tbinom {n-1}{\ell -1}$,
so \eqref{LYM+} gives $|\B_{k}| \geq \tbinom {n-1}{k} + (1 - \tT)\tbinom {n-1}{k-1} 
> \tbinom {n}{k} - |\A|$, which is a contradiction. Thus $|S|=n-1$, as claimed.

Now $|\B_{\ell } \cap \tbinom {S}{\ell }| 
\geq \tbinom {|S|}{\ell } - \tT\tbinom {|S|-1}{\ell -1}
= \tbinom {|S|-1}{\ell } + (1-\tT)\tbinom {|S|-1}{\ell -1}$, 
so $|\B_{k} \cap \tbinom {S}{k}| \geq 
\tbinom {|S|-1}{k} + (1-\tT) \tbinom {|S|-1}{k -1}
= \tbinom {|S|}{k} - \tT \tbinom {|S|-1}{k -1}$ by \eqref{LYM+}.
As $\A \cap \B_k = \es $ this gives 
$|\A \cap \tbinom {S}{k}| \leq \tT \tbinom {n-2}{k-1} 
\leq \tT \tbinom {n-1}{k-1}$. This proves the first statement
of the lemma with the star ${\cal S} := \tbinom{[n]}{k} \sm \tbinom {S}{k}$.

Without loss of generality, 
${\cal S} = {\cal S}_1 = \{ A \in \tbinom{[n]}{k}: 1 \in A\}$.
As $\tT<1/2$ and $n \ge 2k$ we have $E := |\A \sm {\cal S}_1|
 \leq \tT \tbinom{n-1}{k-1} \le \tbinom {n-2}{k-1}$.
Let ${\cal C}: = \{C^c: C \in \A \sm {\cal S}_1\} \subset \tbinom {[n]}{n-k}$. 
Noting that $1 \in C$ for all $C \in {\cal C}$, we take ${\cal C}_{n-k-1} 
:= \{ C: \{1\} \cup C \in {\cal C}\} \subset \tbinom {[2,n]}{n-k-1}$, 
and iteratively define ${\cal C}_i = \pl ({\cal C}_{i+1})$ 
for $n-k-2 \ge i \ge k-1$. Then ${\cal A }\cap {\cal S}_1$
and ${\cal C}_{k-1} + 1$ are disjoint subsets of $\tbinom {[2,n]}{k-1} + 1$,
so $|{\cal A } \cap {\cal S}_1| \leq \tbinom {n-1}{k-1} - |{\cal C}_{k-1}| =  
\tbinom {n-1}{k-1} - |\pl ^{(n-2k)}({\cal C}_{n-k-1})| 
\leq \tbinom {n-1}{k-1} - |\pl ^{(n-2k)}(\I^{(n-k-1)}_{E})|$,
where the last inequality holds by Kruskal--Katona (repeatedly applied).
Thus $|\A| = |\A\cap {\cal S}_1| + |\A\sm {\cal S}_1| 
\leq \tbinom {n-1}{k-1} - |\pl^{(n-2k)}(\I^{(n-k-1)} _{E})| + E = |\F_E|$,
as $\I^{(n-k-1)}_{E} + 1 = \{ A^c: A \in \F^{out}_{E} \} \}$,
so ${\cal S}_1 \sm \F^{in}_{E} = \pl^{(n-2k)}(\I^{(n-k-1)}_{E}) + 1$.
The final statement of the theorem holds as if $E = \tbinom {u}{n-k-1}$ then 
$|\pl ^{(n-2k)}(\I^{(n-k-1)}_{E})| \geq \tbinom {u}{k-1}$ 
by the Lov\'asz form of Kruskal--Katona (repeatedly applied).
\end{proof}		

Next we prove our stability version of Katona's Intersection Theorem.

\begin{proof}[Proof of Theorem \ref{Katona-stability}] 
Suppose  $\A \sub \{0,1\}^n$ is $t$-intersecting,
where $t=2k-n \geq 2$ and $|\A| \geq \tbinom{n}{\ge k}
- \tT\dD \tbinom{n-1}{k-1}$.
Let $\B = \{ A^c: A \in \A \}$. Recall that we denote
iterated neighbourhoods in the cube by $N^i(\cdot)$.
Note that $|N^i(\A)|=|N^i(\B)|$ for any $i \ge 0$,
as $\A$ and $\B$ are isomorphic under the automorphism
of the cube that interchanges $0$ and $1$ in each coordinate.
As $\A \sub \{0,1\}^n$ is $t$-intersecting we have
$N^{t-1}(\A) \cap \B = \es$,
so $|N^{t-1}(\A)| \le 2^n - |\B|
\le \tbinom{n}{\ge n-k+1} + \tT\dD \tbinom{n-1}{k-1}$.

We claim that there is $i < t-1$ with
$|\pl_v(N^i(\A))| < (1 + \tfrac{c}{n}) \tbinom{n}{k-i-1}$,
where $c = 10^{-4} \dD$.
To see this claim, note that if it fails then
$\tbinom{n}{\ge n-k+1} + \tT\dD \tbinom{n-1}{k-1}
\ge |N^{t-1}(\A)| 
\ge |\A| + \sum_{i=0}^{t-2} (1 + \tfrac{c}{n}) \tbinom{n}{k-i-1}
\ge \tbinom{n}{\ge n-k+1} - \tT\dD \tbinom{n-1}{k-1}
+ \tfrac{c}{n} \sum_{i=1}^{t-1} \tbinom{n}{k-i}$.
However, if $t < \sqrt{n}$ we have
$\tfrac{c}{n} \sum_{i=1}^{t-1} \tbinom{n}{k-i}
> 10^{-5} \dD (t-1) n^{-3/2} 2^n 
> 2\tT\dD \tbinom{n-1}{k-1}$
or if $t \ge \sqrt{n}$ we have
$\tfrac{c}{n}\sum_{i=1}^{t-1} \tbinom{n}{k-i}
\ge \tfrac{c}{n} (1 - e^{-t^2/2n} ) 2^{n-1} 
> \tT\dD e^{-t^2/2n} 2^{n-1} 
> 2\tT\dD \tbinom{n-1}{k-1}$.
This contradiction proves the claim.

As $|\A| \geq \tbinom{n}{\ge k} - \tT\dD \tbinom{n-1}{k-1}
= \tbinom{n}{\ge k+1} + \tbinom{n-1}{k} + (1-\tT\dD) \tbinom{n-1}{k-1}$, 
by Harper's Theorem and \eqref{LYM+} we have
$|N^i(\A)| \ge \tbinom{n}{\ge k+1-i} 
+ \tbinom{n-1}{k-i} + (1-\tT\dD) \tbinom{n-1}{k-i-1}
= \tbinom{n}{\ge k-i} - \tT\dD \tbinom{n-1}{k-i-1}$.
Recalling that $|N^i(\A)| \le \tbinom{n}{\ge k-i}$,
by Theorem \ref{thm: ball-sized stability} we have
$|N^i(\A)\triangle \mc{H}_A| \leq 5\tT\dD \tbinom{n-1}{k-i-1}$ 
for some Hamming ball $\mc{H}_A$. Equivalently,
$|N^i(\B)\triangle \mc{H}_B| \leq 5\tT\dD \tbinom{n-1}{k-i-1}$ 
for the Hamming ball $\mc{H}_B = \{A^c: A \in \mc{H}_A\}$.

Write $\mc{H}_B = \B^n_{n-k+i}(C)$ for some $C \in \{0,1\}^n$,
so $\mc{H}_A = \B^n_{n-k+i}(C^c)$, and $\B' = N^i(\B) \cap \mc{H}_B$. 
We have $|\B'| \ge \tbinom{n}{\ge k-i} - 5\tT\dD \tbinom{n-1}{k-i-1}
=  \tbinom{n}{\ge k+1-i} 
+ \tbinom{n-1}{k-i} + (1-5\tT\dD) \tbinom{n-1}{k-i-1}$.
By Harper's Theorem and \eqref{LYM+} we have
$|N^{t-1-i}(\B')| \ge \tbinom{n}{\ge k+2-t} 
+ \tbinom{n-1}{k+1-t} + (1-5\tT\dD) \tbinom{n-1}{k-t}
= \tbinom{n}{\ge k+1-t} - 5\tT\dD \tbinom{n-1}{k-t}$.
As $N^{t-1-i}(\B') \sub N^{t-1}(\B) \cap \B^n_{n-k+t-1}(C)$
and $\A \cap N^{t-1}(\B) = \es$ we deduce
$|\A \sm \B^n_{n-k}(C^c)| = |\A \cap \B^n_{n-k+t-1}(C)|
\le 5\tT\dD \tbinom{n-1}{k-t} = 5\tT\dD \tbinom{n-1}{k-1}$,
so $|\A \triangle \B^n_{n-k}(C^c)| \le 11\tT\dD \tbinom{n-1}{k-1}$.

To prove the first statement of the lemma, it remains to show
$\B^n_{n-k}(C^c) = \tbinom{[n]}{\ge k}$, i.e.\ $C=\es$.
Supposing that $C \ne \es$, we will obtain a contradiction
to $\A$ being $t$-intersecting, by finding $A,A' \in \tbinom{[n]}{k}$ such that $A \triangle C$ and $A' \triangle C$ are in $\A$ with $|(A \triangle C) \cap (A' \triangle C)| \le t-1$. To do so, set $\A' = \{A \in \tbinom{[n]}{k}: A \triangle C \in \A \}$ and note that $|{\cal A}'| \geq (1 - 6\theta \delta ) \binom {n}{k}$. For each $\ell \in [0,|C|]$ let $\binom {[n]}{k,\ell } := \{A \in \binom {[n]}{k}: |A \cap C| = \ell \}$. Note that $\cup _{\ell \in [0,|C|]} \binom {[n]}{k,\ell } = \binom {[n]}{k}$ and a small calculation gives $|\cup _{\ell > |C|/2} \binom {[n]}{k,\ell }| \geq \frac {1}{4} \binom {n}{k}$, as $k\geq n/2$. It follows that $\sum _{\ell > |C| /2} |{\cal A}' \cap  \binom {[n]}{k,\ell }| \geq \sum _{\ell > |C| /2} |\binom {[n]}{k,\ell }| - 
6\tT \dD \tbinom {n}{k} \geq \sum _{\ell > |C|/2} (1 - 24\tT \dD ) |\binom {[n]}{k,\ell }|$. Therefore $|{\cal A}' \cap  \binom {[n]}{k,\ell }| \geq (1 - 24\tT \dD ) |\binom {[n]}{k,\ell }| > \frac {1}{2} 
|\binom {[n]}{k,\ell }| > 0$ for some $\ell > |C|/2$. 
Now consider the graph $G$ on vertex set $\binom {[n]}{k,\ell }$ in which $A_1A_2$ is an edge if $A_1$ and $A_2$ are as disjoint as possible when restricted to both $C$ and $[n]\setminus C$, i.e.\ 
$|A_1 \cap A_2 \cap C| = \max (2\ell - |C|, 0) = 2\ell - |C|$ and $|A_1 \cap A_2 \cap ([n]\setminus C)| = \max (2(k-\ell )- (n-|C|), 0)$. 
Clearly $G$ is regular and non-empty
(we cannot have $\ell=|C|=k$ as this would give
$\es \in \A$, but $\A$ is $t$-intersecting). 
Therefore ${\cal A}' \cap \binom {[n]}{k,\ell }$ contains an edge $A_1A_2 \in E(G)$. But this gives $|(A_1\triangle C) \cap (A_2 \triangle C)| = 
|A_1 \cap A_2 \cap ([n]\setminus C)| = \max (2(k-\ell )-(n-|C|), 0) < t$, since $\ell > |C|/2$. This contradiction gives $C=\es$.

Writing $E = |\A \sm \tbinom{[n]}{\ge k}|$
and $D = |\tbinom{[n]}{\ge k} \sm \A|$,
it remains to show $D \ge E'$. To see this,
suppose for a contradiction that $D<E'$.
By definition of $E'$ we have
$|\pl^{t-1}(I^{(k)}_{\tbinom{n}{k}-D})| > \tbinom{n}{n-k+1}-E$
and $|\pl^{t-1}(I^{(k-1)}_E)| \ge E'$ (otherwise 
$\{A^c: A \in \pl^{t-1}(I^{(k-1)}_E)\} \sub \tbinom{[n]}{k}$
contradicts the definition of $E'$).
Then Lemma \ref{lem: local stability} gives
$|N^{t-1}(\A)| \geq |N^{t-1}(\J_{m,D,E})|
> \tbinom{n}{\ge n-k+1} - E + E'$,
so $|\A|=|\B| \le 2^n - |N^{t-1}(\A)| 
< \tbinom{n}{\ge k}-E'+E < |\A|$, contradiction.
Therefore $D \ge E'$, so $|\A| \le |\G_E|$.
\end{proof}		

Our final application is a stability version of Frankl's bound
for the Erd\H{o}s Matching Conjecture.

\begin{proof}[Proof of Theorem \ref{match-stability}] 
Suppose $\A \sub \tbinom{[n]}{k}$ has no matching of size $t+1$ 
and $|\A| > \tbinom{n}{k} - (1 + \tfrac{rc}{n}) \tbinom{n-t}{k}$.
Let $\A'$ be the set of $A' \in \tbinom{[n]}{k+r}$ that contain 
some $A \in \A$. Then $\A'$ has no matching of size $t+1$,
so $|\A'| \le \tbinom{n}{k+r} - \tbinom{n-t}{k-r}$ by \cite{Frankl-match}.
Let $\B = \tbinom{[n]}{k} \sm \A$ and $\B' = \tbinom{[n]}{k+r} \sm \A'$.
Then $|\B'| \ge \tbinom{n-t}{k+r}$ and $\pl^r(\B') \sub \B$,
so $|\pl^r(\B')| \le \tbinom{n}{k} - |\A| 
< (1+\tfrac{rc}{n}) \tbinom{n-t}{k}$.

We now proceed similarly to the proof of Theorem \ref{EKR-stability}.
We define $\B_{k+r},\dots,\B_k$ by $\B_{k+r} = \B'$
and $\B_i = \pl(\B_{i+1})$ for $k+r>i \ge k$.
We define $x_i \ge k$ by $|\B_i| = \tbinom {x_i}{i}$
and note that $x_{i} \geq x_{i+1}$ for $k+r>i \ge k$.
Then $\tbinom{x_{k+r}}{k+r} = |\B'| \ge \tbinom{n-t}{k+r}$
gives $x_{k+r} \ge n-t$ and $\tbinom{x_k}{k} = |\B_k|
< (1+\tfrac{rc}{n}) \tbinom{n-t}{k}$ gives
$x_k < (1+\tfrac{rc}{kn}) (n-t)$ by Lemma \ref{binomratio}.i.
 
Now we claim that $|\pl(\B_{\ell })| \leq 
(1 + \tfrac{4c}{n}) \tbinom {x_{\ell }}{\ell-1}$ 
for some $\ell \in [k+1,k+r]$. 
Suppose for a contradiction that this fails.
As $x_\ell \ge n-t \ge (t+1)\ell$ for $\ell \le k+r$,
by Lemma \ref{binomratio}.ii we have 
$x_{\ell} \geq (1 + \tfrac{c}{kn}) x_{\ell +1}$. 
However, this implies 
$x_k \ge (1 + \tfrac{c}{kn})^r x_{k+r}
\ge (1 + \tfrac{rc}{kn})(n-t)$,
which contradicts our previous upper bound, so the claim holds.

By Theorem \ref{thm: KK-stab}, there is $S \subset [n]$ 
with $|S| \in \{\bfl{x_\ell},\bcl{x_\ell}\}$ so that 
$|\B_{\ell } \triangle \tbinom {S}{\ell }| \leq \dD \tbinom {|S|-1}{\ell -1}$. 
We claim that $|S| = n-t$. To see this, first note that 
$\tbinom {x_{\ell }}{\ell } \leq \tbinom {|S|}{\ell } + \dD \tbinom {|S|-1}{\ell -1} 
\leq \tbinom {|S| + \dD}{\ell }$ by \eqref{mvt}, 
so $|S| \geq x_{\ell } - \dD > n-t-1$.
On the other hand, if $|S| \ge n-t+1$ then
$|\B_{\ell }| \geq \tbinom{n-t+1}{\ell} - \dD \tbinom {n-t}{\ell-1}
= \tbinom{n-t}{\ell} + (1-\dD) \tbinom {n-t}{\ell-1}$,
so \eqref{LYM+} gives 
$|\B_{k}| \geq \tbinom {n-t}{k} + (1-\dD)\tbinom {n-t}{k-1}$.
As $\dD \leq 1/2$ and $r \le k$ this contradicts the earlier bound
$|\B_{k}| < (1+\tfrac{rc}{n}) \tbinom{n-t}{k}$, so the claim holds.

Now $|\B_{\ell } \cap \tbinom {S}{\ell }| 
\geq \tbinom {|S|}{\ell } - \dD\tbinom {|S|-1}{\ell -1}
= \tbinom {|S|-1}{\ell } + (1-\dD)\tbinom {|S|-1}{\ell -1}$, 
so $|\B_{k} \cap \tbinom {S}{k}| \geq 
\tbinom {|S|-1}{k} + (1-\dD) \tbinom {|S|-1}{k -1}
= \tbinom {|S|}{k} - \dD \tbinom {|S|-1}{k-1}$ by \eqref{LYM+}.
Setting $T = S^c$ and using $r \leq k$, we deduce that
$|\A \triangle {\cal S}_T| < 2\dD \tbinom {|S|-1}{k-1}
+ \tfrac{rc}{n} \tbinom{n-t}{k}
\le 3\dD \tbinom {n-t-1}{k-1}$.\end{proof}

\section{Concluding remarks}

We have obtained tight stability results on various problems
for families that are close to extremal. One consequence of
our stability version of Harper's vertex isoperimetric inequality
is a characterisation of the extremal families for sets of the same size as a generalised Hamming ball;
the latter was independently obtained by Raty \cite{Raty}. 
Our stability result
in the case of ball-sized sets applies to families with vertex boundary
that is within a factor of $1+O(1/n)$ of the minimum possible.
We gave an example to show that the same accuracy of stability
does not hold for larger vertex boundary, but this still leaves open
the question of establishing some stability 
for a wider range of approximations to the minimum.
Recently this has been achieved for ball-sized sets, 
where the ball has radius $o(\log n)$, 
in independent work (with a different proof technique) 
by Przykucki and Roberts (personal communication). 

We would be particularly interested in knowing
the level of isoperimetric approximation required for 
stability in the dense case (families of size $\Omega (2^n)$);
we believe that the following may be true.

\begin{conjecture}
	Given $\eps > 0$ there is $\delta > 0$ such that the 
	following holds. Suppose ${\cal A} \subset \{0,1\}^n$ 
	with $|{\cal A}| \geq \eps 2^n$ 
	and $|{\partial }_v({\cal A})| \leq \big (1 + 
	\frac {\delta }{\sqrt n} \big ) \Lov {{\cal A}}$. 
	Then $|{\cal A} \triangle {\cal H}| \leq 
	\eps |{\cal A}|$ for some Hamming ball ${\cal H}$.
\end{conjecture}

\noindent If true this dependence would be tight, as shown by taking ${\cal A} = {\cal H} \times \{0,1\}^d$ where ${\cal H}$  is a Hamming ball of size $2^{n-d-1}$ (say) in $\{0,1\}^{n-d}$ with $d = \Theta _{\eps }(n^{1/2})$.

\end{document}